\documentclass[11pt]{article}
\usepackage[utf8]{inputenc}
\usepackage{microtype}
\usepackage{bold-extra}
\usepackage{chngcntr}
\usepackage{tikz}
\usepackage{setspace}
\usepackage{url}
\usepackage{amsfonts}
\usepackage{epsfig}
\usepackage{commath}
\usepackage{natbib}
\usepackage{graphicx}
\usepackage{bbm, float}
\usepackage{dsfont}
\usepackage{color}
\usepackage{adjustbox}
\usepackage{amsmath, amssymb, amstext, amsfonts, amsthm, latexsym}
\usepackage{diagbox}
\newcommand{\bm}[1]{\boldsymbol{#1}}

\newif\ifshowgrammar
\showgrammartrue 

\usepackage{caption}
\usepackage{subcaption}
\usepackage{multirow}
\usepackage{color}
\usepackage{listings}
\usepackage{titlesec}
\usepackage{enumerate}
\usepackage{comment}
\usepackage{environ}
\usepackage{booktabs}
\usepackage{appendix}
\usepackage{bbm}
\usepackage{caption}
\usepackage{pgfplots}
\usepackage{makecell}
\usepackage{bm}
\usepackage{threeparttable}

\usepackage[linesnumbered, ruled, vlined]{algorithm2e}
\pgfplotsset{compat=1.16}
\definecolor{legcol1}{HTML}{1B6CA8}
\definecolor{legcol2}{HTML}{C0392B}
\definecolor{legcol3}{HTML}{27AE60}
\definecolor{legcol4}{HTML}{8E44AD}
\usepackage{standalone}
\usepackage[margin=1in]{geometry}
\usepackage[colorlinks=true, linkcolor=blue, citecolor=blue, urlcolor=blue]{hyperref}
\theoremstyle{plain}
\newtheorem{theorem}{Theorem}
\newtheorem{corollary}{Corollary}
\newtheorem{proposition}{Proposition}

\newtheorem{lemma}{Lemma}

\theoremstyle{definition}

\newtheorem{definition}{Definition}
\newtheorem{example}{Example}
\newtheorem{assumption}{Assumption}

\theoremstyle{remark}

\newtheorem{remark}{Remark}

\title{Sharp bounds for products of dependent random variables}

\usepackage{authblk}
\author[1]{Christopher Blier-Wong}
\author[2]{Jinghui Chen}

\affil[1]{Department of Statistical Sciences, University of Toronto, Canada}
\affil[2]{Department of Mathematics, University of Connecticut, USA}

\date{\today}

\renewcommand{\d}{\mathrm{d}}
\newcommand{\id}{\mathbbm{1}}
\newcommand{\p}{\mathbb{P}}
\newcommand{\E}{\mathbb{E}}
\newcommand{\R}{\mathbb{R}}
\newcommand{\U}{\mathrm{U}}

\newcommand{\eqiid}{\overset{\mathrm{iid}}{=}}

\DeclareMathOperator{\conv}{conv}

\begin{document}
\maketitle

\begin{abstract}
We study the sharp bounds on $\mathbb{E}[X_1\cdots X_d]$ when the univariate marginal distributions are known, but the dependence structure between them is unspecified. Maximizing products over non-negative variables is straightforward via the comonotonic coupling, but the problem is more subtle when the marginals can take both positive and negative values. Specifically, two negative realizations can be matched to yield a positive product, whereas a single negative realization necessarily yields a negative product. We decompose the problem into a magnitude part and a sign part, and show that universal upper and lower bounds for the product expectation follow from comonotonic coupling of the absolute values and properly chosen sign vectors. Under a mild regularity assumption, we give necessary and sufficient conditions for these universal bounds to be attainable. For the upper bound, the marginal sign-bias vector must belong to the even-parity polytope, while for the lower, the corresponding condition involves the odd-parity polytope. We construct the extremal couplings via measurable selections on the parity polytope whenever these conditions hold. We study the case of identical marginals in more detail and provide examples of nonsymmetric extremal couplings that attain the universal bounds. We explicitly construct the extremal copulas in three dimensions, and use a recursive parity decomposition to obtain higher-dimensional extremal copulas from the trivariate ones. 
\end{abstract}


\section{Introduction}\label{sec:intro}
The marginal distribution of a random variable is often relatively easy to estimate. For example, the distribution of returns on a company's stock can be inferred from its historical price data, and such an estimate generally becomes more precise as more data are collected over time. In contrast, estimating the dependence structure among marginal distributions is much more challenging \citep{mcneil2015quantitative, wang2015extreme}. Therefore, we need tools to understand dependence uncertainty.

A classical dependence uncertainty problem is to determine bounds on the expectation of $\psi(X_1, \dots, X_d)$, where $\psi : \mathbb{R}^d \rightarrow \mathbb{R}$ is a measurable function, that is,
\begin{subequations}
	\label{sharp bounds}
	\begin{align}
		m_d^{\psi}&=\inf\{\mathbb{E}[\psi(X_1, \dots, X_d)];\ X_i\sim F_i, \ 1\leq i\leq d\}, \label{LB} \\
		M_d^{\psi}&=\sup\{\mathbb{E}[\psi(X_1, \dots, X_d)];\ X_i\sim F_i, \ 1\leq i\leq d\}. \label{UB}
	\end{align}
\end{subequations}
These two optimization problems arise in a wide range of areas, including operations research, statistics, and quantitative risk management. We refer to \citet{puccetti2012computation} for applications in statistics, to \citet{embrechts2010risk} for quantitative risk management, and to \citet{hsu1984approximation} and \citet{jakobsons2016negative} for applications in operations research. We also refer to \citet{ruschendorf2024model} for a detailed overview. Many important instances of \eqref{LB} and \eqref{UB} have been studied in the literature. For supermodular $\psi$, \eqref{UB} admits a general analytical solution: whenever the relevant expectations exist, it is attained by the comonotonic coupling \citep{tchen1980inequalities, ruschendorf1983solution}. In contrast, for $d\geq 3$, no single coupling solves \eqref{LB} for every supermodular $\psi$; see \citet{puccetti2015extremal} for these results and their historical development. An important special case is $\psi(x_1, \dots, x_d)=f\left(\sum_{i=1}^d x_i\right)$ with $f$ convex, which arises in risk aggregation. The comonotonic coupling yields the upper bound, while a constant-sum coupling yields the Jensen lower bound whenever the marginals are jointly mixable, which is extended from complete mixability in the case of identical marginals \citep{wang2011complete, wang2016joint}. For aggregate-tail indicators $\psi(x_1, \dots, x_d)=\mathds{1}_{\{x_1+\cdots+x_d>x\}}$, $x\in\mathbb{R}$, bounds and sharpness results are given by \citet{embrechts2006bounds} and \citet{puccetti2013sharp}. A related conditional-tail functional is the marginal expected shortfall $\E[X_i\mid S>F^{-1}(p)]$, with $p\in(0, 1)$ and $S=\sum_{i=1}^dX_i\sim F$, for which \citet{chen2026static} derive sharp bounds. When analytical solutions are unavailable, \citet{puccetti2012computation} introduce the rearrangement algorithm for distributional bounds, based on the rearrangement representation of \citet{ruschendorf1983solution}, and \citet{puccetti2015computation} extend it to expected supermodular functionals.

The expectation of the product of random variables, that is,  $\psi(X_1, \dots, X_d):=\prod_{i=1}^{d}X_i$, is called a mixed moment. Mixed moments play an important role in multivariate analysis. In particular, they encode higher-order dependence beyond correlation, appear in cumulant expansions, and underlie commonly used measures of statistical association such as covariance, coskewness, and cokurtosis. The product expectation is therefore a basic building block for quantifying dependence in a multivariate random vector. It captures interaction effects that may not be visible through other measures of multivariate association. Mixed moments also arise in cumulant approximations of nonlinear portfolio functionals in risk management and finance. For example, the mean--variance efficient frontier \citep{markowitz1952portfolio} and the capital asset pricing model (CAPM) \citep{sharpe1964capital} rely on covariance to describe systematic risk. More recent studies have also focused on coskewness as a risk measure, defined as $\E[\prod_{i=1}^3\frac{X_i-\mu_i}{\sigma_i}],$ where $\mu_i$ and $\sigma_i$ are the marginal means and standard deviations of the marginal distributions, for $i \in \{1, 2, 3\}.$ Coskewness measures the comovement of three rvs: positive coskewness implies all rvs undergo extreme positive deviations, or one undergoes a positive deviation while the others undergo negative deviations. Negative coskewness implies all negative deviations, or that one undergoes positive deviations. \citet{harvey2000conditional} show that incorporating coskewness into asset pricing models helps explain the cross-sectional variation in expected returns, especially in settings where the traditional CAPM performs poorly.

Figure~\ref{fig:rolling-three-stock-coskewness} presents the coskewness of three financial institutions, JPMorgan Chase (JPM), Bank of America (BAC), and Citigroup (C), over time, computed using a 60-trading-day rolling window. The returns are calculated from adjusted closing prices from Yahoo Finance for January 2005 to December 2025. The observed series changes sign frequently and ranges from approximately $-2.5$ to $2.5$. Unlike Pearson correlation, which is universally bounded between $-1$ and $1$, raw coskewness has no such universal range. For each portfolio and each window, its feasible range over all dependence structures depends on the marginal distribution of the returns. The raw values therefore cannot be compared directly across different marginal distributions and portfolios or over time, nor do they reveal where the observed three-dimensional comovement lies within its feasible range. This distinction is economically relevant because, for a portfolio without short selling, smaller coskewness reduces the term involving all three assets in the portfolio's third central moment and therefore represents a more adverse dependence contribution. These limitations motivate our study of lower and upper bounds on product expectations. One of our main contributions, developed in Section~\ref{sec:coskewness-application}, is to use the corresponding lower and upper coskewness bounds to normalize the raw coskewness on the fixed interval $[-1, 1]$. The resulting standardized measure maps coskewness to a fixed range, allowing the analyst to meaningfully compare the third product moment across different marginal distributions and portfolios and over time.

\begin{figure}[t]
  \centering
  \includegraphics[width=\textwidth]{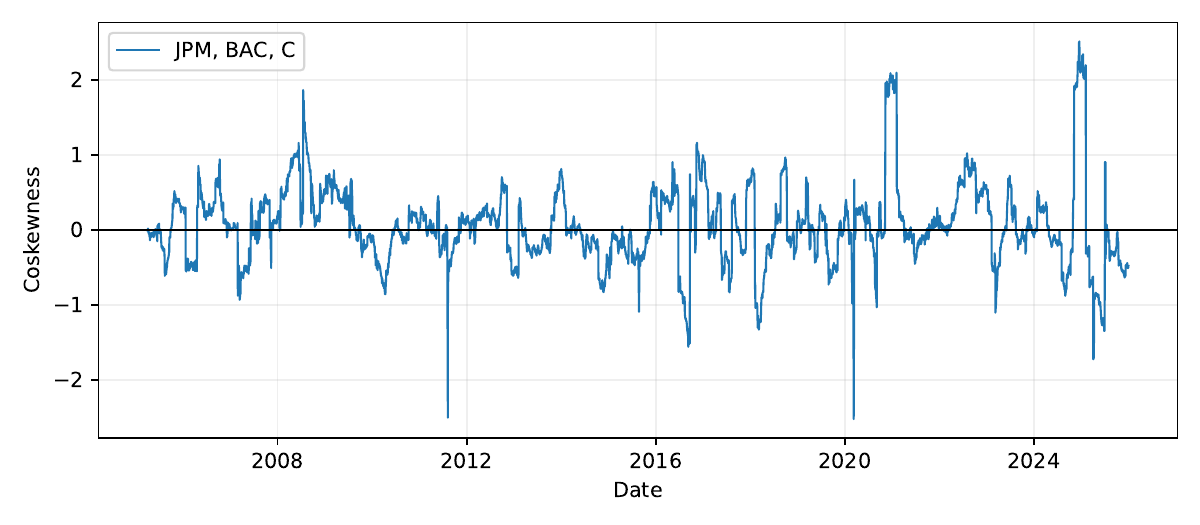}
  \caption{Rolling coskewness of daily returns on three financial institutions computed over 60-trading-day windows.}
  \label{fig:rolling-three-stock-coskewness}
\end{figure}

In this paper, we study sharp bounds on $\mathbb{E}[\prod_{i=1}^{d}X_i]$ when the marginal distributions $F_i$ are known but the dependence structure is unknown, i.e., when $\psi(X_1, \dots, X_d):=\prod_{i=1}^{d}X_i$ in the definitions of $m_d^\psi$ and $M_d^\psi$. Henceforth, $m_d^\times$ and $M_d^\times$ denote the solutions to the minimization and maximization problems, respectively. The problems of determining $m_d^{\times}$ and $M_d^{\times}$ fall under the classical theory of distributions with given marginals, also known as the Fr\'echet problems. The existence of couplings with prescribed marginals under considerable generality is guaranteed by \citet{strassen1965existence}, while a duality theory for optimizing $\int h \,\d\mu$ over a Fr\'echet class is developed by \citet{kellerer1984duality}. For $d=2$, the problems are classical, and the extremizers are given by the comonotonic and countermonotonic couplings. For $d\ge 3$, however, the problem becomes more delicate. The product of nonnegative variables is supermodular, and therefore the comonotonic coupling attains $M_d^{\times}$ \citep[see, e.g.,][]{shaked2007stochastic}. In contrast, there is no single broadly applicable multivariate analogue of countermonotonicity; see, for example, \cite{joag-dev1983negative}, \cite{puccetti2015extremal}, \cite{wang2015extreme, wang2016joint}, \cite{lauzier2023pairwise} and \cite{cossette2025extremal} for different notions of negative dependence. For $M_3^\times$ and $m_3^\times$ with $X_i\sim \U[0, 1]$, \citet{nelsen2012directional} obtain related bounds and connect them to multivariate dependence measures. For $X_i\sim \U[a_i, b_i]$ with $b_i>a_i\ge 0$, additional closed-form expressions for $m_d^{\times}$ and $M_d^{\times}$ are derived by \citet{bignozzi2015studying}. Once the marginals take both positive and negative values, determining $m_d^{\times}$ and $M_d^{\times}$ becomes substantially harder. For $m_d^{\times}$, \citet{wang2016joint} propose a solution when $(F_1, \cdots, F_d)$ is jointly mixable, a concept extending the complete mixability introduced by \citet{wang2011complete}.

More recently, \citet{bernard2023coskewness} provide sharp characterizations of $m_d^{\times}$ and $M_d^{\times}$, and identify their dependence structures, called cross-product copulas, for symmetric marginals with zero means. However, they do not explain why the randomizers in these copulas follow a Bernoulli distribution with success probability $0.5$. Moreover, the optimization problems for asymmetric marginals or symmetric distributions with non-zero means remain open.

In this paper, we complete the attainability analysis of the universal product bounds for general signed marginals. The difficulty is that if the marginals can take both positive and negative values, the product is generally neither supermodular nor submodular, so the standard dependence-order arguments used for many aggregation problems do not apply. We identify the missing structure through a sign-recombination mechanism. The key idea is that negative parts must be combined across coordinates so that the sign of the product is controlled. For example, in three dimensions, a negative realization in one coordinate should typically be matched with another negative realization if one wants to obtain a positive product, whereas one negative sign or three negative signs push the product in the opposite direction. The extremal coupling must therefore coordinate both signs and magnitudes. This sign-recombination mechanism yields a parity-polytope characterization of attainability and a construction of extremal couplings.

Motivated by the sharp characterization in \citet{bernard2023coskewness}, we decouple the problem of finding sharp bounds for the signed product into a magnitude part and a sign part. The magnitudes are nonnegative, so their maximizing coupling is governed by comonotonicity. The remaining question is whether the signs induced by the marginals can be recombined across coordinates so that the product is nonnegative (resp.\ nonpositive) at each magnitude level for $M_d^\times$ (resp.\ $m_d^\times$). This sign-compatibility problem is the main technical challenge in the general case, and it is also the key to understanding the structure of the extremal couplings. To resolve this problem, we formulate sign compatibility in terms of the even- and odd-parity polytopes. These polytopes also arise in polyhedral combinatorics and coding theory \citep{jeroslow1975defining, barman2013decomposition, ermel2020parity}. This formulation converts attainability of the universal product bounds into a pointwise polytope-membership problem, from which we derive sharpness conditions and construct extremal couplings. The resulting characterization is directly implementable from the marginals: at each magnitude level, one computes the conditional positive-sign probabilities and checks a finite collection of linear inequalities. Feasibility certifies that the universal bound is exact, and the corresponding polytope weights construct an extremal stress coupling.

Our paper makes the following contributions to the literature. First, we identify parity feasibility as the geometric structure governing the signed-product problem and use this formulation to establish necessary and sufficient conditions to attain the universal upper and lower bounds in arbitrary dimensions. We then use the feasible polytope weights to construct the maximizing and minimizing couplings explicitly. Second, for identically distributed marginals, we translate the general attainability criterion into explicit dimension-dependent conditions. We show that, as $d$ increases, the feasibility band widens and, when $d$ is even, $m_d^\times$ tolerates progressively more asymmetry. Third, we explicitly construct the maximizing and minimizing copulas that attain the universal upper and lower bounds for $M_3^\times$ and $m_3^\times$, respectively. This construction clarifies the role of the randomizers in these optimizing copulas. In particular, it explains why the break randomizers in the cross-product copulas of \citet{bernard2023coskewness} follow a Bernoulli distribution with success probability $0.5$ for symmetric marginals with zero means. Finally, we propose a recursive method for constructing the maximizing and minimizing copulas when $d\ge 4$.

The remainder of the paper is organized as follows. Section \ref{sec:notation} collects background material, notation, and the universal absolute-value bounds from the literature. Section \ref{sec:sharp-max} presents the parity-feasibility characterizations and the piecewise-affine selection constructions for both the upper and lower bounds. Section \ref{sec:iid} studies the case of identical marginals, deriving scalar threshold conditions and a sparsity bound on the number of active sign patterns. Section \ref{sec:d3} specializes both the upper and lower bound results to the trivariate case, while Section \ref{sec:recursive-parity} introduces the recursive method for constructing the maximizing and minimizing copulas and provides an illustrative example for $d=4$. Section \ref{sec:numerical} numerically verifies sharpness and illustrates the gaps that arise when the parity conditions fail. Section \ref{sec:coskewness-application} applies the trivariate results to coskewness under asymmetric return marginals. The final section concludes, and Appendix \ref{sec:examples} provides additional examples. Some longer proofs are deferred to Appendix \ref{app:proofs}.

\section{Background and preliminaries}\label{sec:notation}

We introduce notation, recall some known results, and develop the universal absolute-value bounds from \citet{bernard2023coskewness} that the remainder of the paper builds on.

\subsection{Preliminaries}

Let $(\Omega, \mathcal{F}, \mathbb{P})$ denote a standard atomless probability space, and let $\mathcal{X} := L^0(\Omega, \mathcal{F}, \mathbb{P})$ be the space of all real-valued random variables (rvs) defined on this space. 
We denote generic elements of $\mathcal{X}$ by $X, \ U, \ V, \ X_i,$ and $Y_i$, where $i \in [d] := \{1, 2, \ldots, d\}$ and $d \in \mathbb{N}\setminus \{1\}$. The notation $U\eqiid V$ means that $U$ and $V$ are independent and identically distributed. Unless stated otherwise, $U \sim \U[0, 1]$ is assumed to be independent of all other rvs. 
Throughout the paper, the notation $X \sim F_X$ indicates that the rv $X$ has cumulative distribution function (cdf) $F_X(x) = \mathbb{P}(X \le x)$, $x \in \mathbb{R}$. Moreover, $\bar F_X(x)=1-F_X(x)$ denotes the survival function of $F_X$. The left-continuous quantile, or generalized inverse function of $X$, for $u \in (0, 1)$, is defined by  
$$F_X^{-1}(u) := \inf\{x \in \mathbb{R} : F_X(x) \ge u\}.$$
This convention ensures that $F_X^{-1}(U)$ has cdf $F_X$. Let $\id_A$ denote the indicator function of an event $A$. Throughout, $F_i$ denotes the cdf of $X_i$, and we reserve the notation $G_i$ for the cdf of $Y_i:=|X_i|$.

Vectors are denoted in boldface, e.g., $\bm u=(u_1, \dots, u_d)$, and inequalities between vectors are understood coordinatewise. For a random vector $\bm X = (X_1, \ldots, X_d)$, we denote its joint cdf by  
$$F_{\bm X}(\bm x) = \mathbb{P}(\bm X \le \bm x), \text{ where } \bm x = (x_1, \ldots, x_d) \in \mathbb{R}^d.$$ Sklar's theorem \citep{sklar1959fonctions} states that the joint cdf $F_{\bm X}$ of a random vector $\bm X=(X_1, \dots, X_d)$ with marginal cdfs $F_i$, $i\in[d]$, can be expressed using a copula $C:[0, 1]^d\rightarrow[0, 1]$ as follows: $$F_{\bm X}(\bm x)=C(F_1(x_1), \dots, F_d(x_d)).$$ We use copulas to study dependence structures and refer the reader to \cite{nelsen2007introduction} and \cite{joe2014dependence} for detailed introductions to copulas. 

Throughout the paper, $\E[|X_1\cdots X_d|]<\infty$ when the product moment is discussed; a sufficient condition is $\E[|X_i|^d]<\infty$ for every $i\in [d]$ by H\"older's inequality. When $\E[|X_1\cdots X_d|]<\infty$, by Sklar’s Theorem \citep{sklar1959fonctions}, we have the representation
$$\E\left[X_1\cdots X_d\right]=\int_{[0, 1]^d}\ \prod_{i=1}^d F_i^{-1}(x_i)\d C(x_1, \dots, x_d),$$
which will be exploited later.

A random vector $\bm X = (X_1, \ldots, X_d)$ with $X_i \sim F_i$ for $i\in [d]$ is said to be comonotonic if there exists a single $U\sim\U[0, 1]$ such that
$$(X_1, \dots, X_d)\stackrel{\mathrm{d}}{=}\left(F_1^{-1}(U), \dots, F_d^{-1}(U)\right);$$
see Proposition 4.5 of \cite{denneberg1994nonadditive}.

In this paper, the central goal is to derive lower and upper bounds for the expected product of $d\geq 2$ rvs with known marginal distributions but unspecified dependence between them. That is, we consider the Fréchet class of all joint cdfs with fixed marginals. Specifically, we consider the problems
\begin{subequations}
	\label{sharp product bounds}
\begin{align}
m_d^{\times}  &:= \inf\left\{\E\left[X_1\cdots X_d\right]:\ X_i\sim F_i, \ i=1, \dots, d\right\}, \label{eq:md}\\
M_d^{\times}  &:= \sup\left\{\E\left[X_1\cdots X_d\right]:\ X_i\sim F_i, \ i=1, \dots, d\right\}.\label{eq:Md}
\end{align}
\end{subequations}

When $d=2$, the function $x_1 x_2$ is supermodular for all $(x_1, x_2) \in \mathbb{R}^2$, so the comonotonic coupling attains $M_2^{\times}$, with value $\E[F_1^{-1}(U)F_2^{-1}(U)]$, whereas the countermonotonic coupling attains $m_2^{\times}$, with value $\E[F_1^{-1}(U)F_2^{-1}(1-U)]$. For $d\ge3$, sharp bounds are much more delicate even for simple marginals. Explicit solution of either \eqref{eq:md} or \eqref{eq:Md} is largely open for symmetric marginals with non-zero means and the asymmetric case, which motivates the results of this paper. 

\subsection{Universal absolute-value bounds}\label{sec:absolute-bounds}
The following lemma follows from Theorem~2.1 in \citet{puccetti2015extremal}. The product function is supermodular on $\mathbb{R}_+^d$ and strictly supermodular on $(0, \infty)^d$; see Section~7.A.8 of \citet{shaked2007stochastic}. Since continuity of each $G_i$ implies $Y_i>0$ a.s., the strict case applies under the additional continuity and finiteness conditions in the lemma. We therefore omit the proof.

\begin{lemma}\label{lem:nonneg-product}
Let $Y_i\sim G_i$, $i\in [d]$, be any nonnegative rvs, and let $U\sim \U[0, 1]$. Then
\begin{equation}\label{eq:nonneg-upper}
\E\left[\prod_{i=1}^d Y_i\right] \le \E\left[\prod_{i=1}^d G_i^{-1}(U)\right]=\int_0^1 \prod_{i=1}^d G_i^{-1}(u)\d u.
\end{equation}
Moreover, equality holds if $(Y_1, \dots, Y_d)$ is a comonotonic vector, i.e., $Y_i=G^{-1}_i(U)$.
If in addition all $G_i$ are continuous and $\int_0^1 \prod_{i=1}^d G_i^{-1}(u) \d u<\infty$, then equality in \eqref{eq:nonneg-upper} implies $(Y_1, \dots, Y_d)$ is comonotonic.
\end{lemma}

We now turn to signed products, where the pointwise inequalities $-\prod_{i=1}^d |X_i|\le \prod_{i=1}^d X_i\le \prod_{i=1}^d |X_i|$ combine with Lemma~\ref{lem:nonneg-product} to produce universal two-sided bounds. The bounds and sufficient conditions for equality are derived in Lemmas~3.1 and~3.2 of \citet{bernard2023coskewness}; the necessary equality conditions follow by combining these pointwise inequalities with Lemma~\ref{lem:nonneg-product}.

\begin{lemma}\label{lem:product-bounds}
Let $X_i\sim F_i$ be such that $|X_i|\sim G_i$, for $i\in [d]$, and let $U\sim \U[0, 1]$. Then
\begin{align}
 M_d^{\times} &\le \E\left[\prod_{i=1}^d G_i^{-1}(U)\right],\label{eq:upper-bound}\\
 m_d^{\times} &\ge -\E\left[\prod_{i=1}^d G_i^{-1}(U)\right].\label{eq:lower-bound}
\end{align}
Moreover, equality holds in \eqref{eq:upper-bound} (resp.\ \eqref{eq:lower-bound}) if $|X_1|, \dots, |X_d|$ are comonotonic and $\prod_{i=1}^d X_i\ge 0$ (resp.\ $\prod_{i=1}^d X_i\le 0$) a.s.

If the right-hand side of \eqref{eq:upper-bound} is finite, then equality in \eqref{eq:upper-bound} (resp.\ \eqref{eq:lower-bound}) implies that any coupling $(X_1, \dots, X_d)$ attaining $M_d^{\times}$ (resp.\ $m_d^{\times}$) has a nonnegative (resp.\ nonpositive) product a.s.\ and satisfies
$$\E\left[\prod_{i=1}^d |X_i|\right]=\E\left[\prod_{i=1}^d G_i^{-1}(U)\right].$$
If, in addition, all $G_i$ are continuous, then the absolute values $|X_1|, \dots, |X_d|$ in any such coupling are comonotonic.
\end{lemma}

The necessary conditions in Lemmas~\ref{lem:nonneg-product} and~\ref{lem:product-bounds} provide the equality characterizations used throughout the remainder of the paper: when all $G_i$ are continuous and the universal bound is finite, attaining either signed-product bound forces the absolute values to be comonotonic. They narrow the attainability question from a search over all copulas to a finite-dimensional sign-coordination problem that we characterize in the next section.

\section{Sharp product bounds}\label{sec:sharp-max}
In this section, we present our main result. Lemma~\ref{lem:product-bounds} shows that, under its continuity and finiteness conditions, any coupling attaining the universal upper (resp.\ lower) bound must have comonotonic absolute values and a nonnegative (resp.\ nonpositive) product a.s. In both cases, after fixing $|X_i|=G_i^{-1}(U)$ for $i\in [d]$, the remaining issue is so-called sign compatibility. That is, whether, for each quantile level $u$, the marginal sign biases can be realized by a distribution supported on even-parity sign vectors for the upper bound and on odd-parity sign vectors for the lower. We address the upper bound first and then develop the odd-parity counterpart.

Recall that the right-hand side of \eqref{eq:upper-bound} is the value for the magnitudes under comonotonicity between $|X_i|$. The main result of this section identifies exactly when $M_d^{\times}$ and $m_d^{\times}$ are sharp and uses a measurable selection argument to turn pointwise parity feasibility into an actual maximizing and minimizing coupling.

\subsection{Sign bias and parity polytope}\label{subsec:sign-bias}

The universal upper bound of $M_d^{\times}$ in \eqref{eq:upper-bound} couples absolute values comonotonically, but says nothing about the signs that each marginal takes. Whether the bound is attainable depends on whether, at every quantile level $u$, the marginal sign combinations can be realized simultaneously.

\begin{assumption}\label{ass:regularity}
For each $i\in[d]$, the marginal $F_i$ is absolutely continuous with connected support and density $f_i$ that is positive on the interior of the support of $F_i$, and $\E\left[|X_i|^d\right]<\infty$ for $X_i\sim F_i$.
\end{assumption}

For $u\in[0, 1]$, define the function
\begin{equation}\label{eq:Hi-def}
H_i(u) := \p\bigl(X_i\ge 0, G_i(|X_i|)\le u\bigr).
\end{equation}
We define the sign bias $p_i(u)$ as the a.e.\ derivative of $H_i(u)$. Under Assumption~\ref{ass:regularity}, $F_i$ has density $f_i$, hence $\p(X_i=0)=0$, and the sign of $X_i$ is well-defined a.s.

\begin{lemma}\label{lem:sign-bias}
Under Assumption~\ref{ass:regularity}, the function $H_i$ defined in \eqref{eq:Hi-def} is increasing and $1$-Lipschitz on $[0, 1]$, hence absolutely continuous. Its derivative $p_i\in[0, 1]$ satisfies, for a.e.\ $u\in(0, 1)$,
\begin{equation}\label{eq:sign-bias}
p_i(u) = \frac{f_i(G_i^{-1}(u))}{f_i(G_i^{-1}(u))+f_i(-G_i^{-1}(u))}.
\end{equation}
\end{lemma}

\begin{proof}
For $0\le a<b\le 1$, we have
$$0 \le H_i(b)-H_i(a) = \p\bigl(X_i\ge 0, a<G_i(|X_i|)\le b\bigr) \le \p\bigl(a<G_i(|X_i|)\le b\bigr) = b-a,$$
where the final equality uses $G_i(|X_i|)\sim \U[0, 1]$ under Assumption~\ref{ass:regularity}. Then $H_i$ is increasing and $1$-Lipschitz. Thus it is absolutely continuous with an a.e.\ derivative $p_i\in[0, 1]$ satisfying $H_i(u)=\int_0^u p_i(r)\d r$.

Since $F_i$ has density $f_i$, the rv $|X_i|$ has density $g_i(y)=f_i(y)+f_i(-y)$ on $(0, \infty)$, and $\p(X_i\ge 0, |X_i|\in B)=\int_B f_i(y)\d y$ for every Borel $B\subset(0, \infty)$. Therefore
$$H_i(u)=\p(X_i\ge 0, |X_i|\le G_i^{-1}(u))=\int_0^{G_i^{-1}(u)} f_i(y)\d y.$$
Since the support of $F_i$ is an interval and $f_i$ is positive on its interior by Assumption~\ref{ass:regularity}, the support of $|X_i|$ is also an interval and $g_i(y)=f_i(y)+f_i(-y)$ is positive on its interior. Hence $G_i$ is strictly increasing on its support and the derivative $\partial G_i^{-1}(r)=1/g_i(G_i^{-1}(r))$ for a.e.\ $r$. Changing variables with $y=G_i^{-1}(r)$, we obtain
$$H_i(u)=\int_0^u \frac{f_i(G_i^{-1}(r))}{f_i(G_i^{-1}(r))+f_i(-G_i^{-1}(r))}\d r,$$
and \eqref{eq:sign-bias} follows by differentiation. Since the density $f_i$ may be chosen to be Borel-measurable and the quantile function $G_i^{-1}$ is left-continuous (hence Borel), the right-hand side of \eqref{eq:sign-bias} provides a Borel-measurable representative of $p_i(u)$.
\end{proof}

Although the formal development requires only that $p_i(u)$ is the a.e.\ derivative of $H_i(u)$, the interpretation of the function $p_i(u)$ as $\p(X_i\ge 0\mid |X_i|=G_i^{-1}(u))$ is very useful. This interpretation motivates us to introduce some geometric concepts to address the attainability of the universal upper (resp.\ lower) bound in \eqref{eq:upper-bound} (resp.\ \eqref{eq:lower-bound}).

\begin{definition}[Parity Sign Set]
Let $\bm s=(s_1, \dots, s_d)$ be a random vector with $s_i\in \{-1, +1\}$ for $i\in [d]$. The even-parity and odd-parity sign sets, denoted by $E_d$ and $O_d$, are defined as
$$E_d := \left\{\bm s\in\{-1, +1\}^d:\ \prod_{i=1}^d s_i=+1\right\}$$
and $$O_d := \left\{\bm s\in\{-1, +1\}^d:\ \prod_{i=1}^d s_i=-1\right\},$$
respectively.
\end{definition}Clearly,
$|E_d|=|O_d|=2^{d-1}.$ For example, by omitting the number 1 (only looking at the signs), we have
\begin{equation}\label{eq: travariate even sign set}
    E_3=\{(+, +, +), (+, -, -), (-, +, -), (-, -, +)\},
\end{equation}and 
\begin{equation}\label{eq: travariate odd sign set}
    O_3=\{(+, +, -), (+, -, +), (-, +, +), (-, -, -)\}.
\end{equation}Moreover, each element of $O_3$ can be obtained from the corresponding element of $E_3$ by flipping exactly one sign, i.e., $(s_1, s_2, s_3)\in E_3$ implies that $(s_1, s_2, -s_3)\in O_3.$
For example, $(+1, +1, +1)\in E_3$ and $(+1, +1, -1)\in O_3.$

\begin{definition}[Parity Polytope]
For each $\bm s\in\{-1, +1\}^d$ define its indicator vector
$$v(\bm s) := \left(\id_{\{s_1=+1\}}, \dots, \id_{\{s_d=+1\}}\right)\in\{0, 1\}^d.$$
The even-parity and odd-parity polytopes are respectively
\begin{equation}\label{eq:parity-polytope}
P_d^+ := \mathrm{conv}\left\{v(\bm s): \bm s\in E_d\right\} \subset [0, 1]^d,
\end{equation}
and
\begin{equation}\label{eq:odd-parity-polytope}
P_d^- := \conv\left\{v(\bm s): \bm s\in O_d\right\} \subset [0, 1]^d.
\end{equation}
\end{definition}

Parity polytopes encode binary vectors subject to an even- or odd-cardinality constraint. In polyhedral combinatorics, their linear descriptions allow parity restrictions to be incorporated into optimization models \citep{jeroslow1975defining, ermel2020parity}; in coding theory, they are used in linear-programming decoding, where efficient projection onto a parity polytope enables scalable algorithms \citep{barman2013decomposition}. In our setting, the vertices represent sign patterns whose product has a prescribed sign, while a point in the polytope records the marginal probabilities of positive signs. Membership of the sign probability vector to a parity polytope therefore determines whether they can be combined into a joint sign distribution attaining the upper or lower product bound, and the corresponding weights provide the mixing probabilities used to construct an extremal coupling. 

Similarly, following \citet{jeroslow1975defining} and \citet{ermel2020parity}, it is convenient to represent $P_d^+$ and $P_d^-$ using forbidden-set inequalities. This representation turns polytope membership into a finite system of linear inequalities and yields the explicit trivariate feasibility conditions in Section~\ref{sec:d3}. For $\bm p=(p_1, \dots, p_d)\in[0, 1]^d$, the inequalities take the form
\begin{equation}\label{eq:parity-forbidden-set}
\sum_{i\notin F}(1-p_i)+\sum_{i\in F}p_i\ge1.
\end{equation}
Then $\bm p\in P_d^+$ iff \eqref{eq:parity-forbidden-set} holds for every odd-cardinality set $F\subseteq[d]$, whereas $\bm p\in P_d^-$ iff it holds for every even-cardinality set $F\subseteq[d]$.

If $\bm s =(+1, +1, -1)$, we have $v(\bm s) = \left(1, 1, 0\right).$ Note that each $p_i(u)$ is the probability that $X_i$ is positive given its absolute value rank $u$. A sign pattern $\bm s\in E_d$ assigns a sign $s_i\in\{-1, +1\}$ to every coordinate. The polytope $P_d^+$ collects all vectors $\left(p_1, \dots, p_d\right)\in[0, 1]^d$ that can arise as the probabilities $p_i=\p(s_i=+1)$ under some distribution on even-parity sign patterns $\bm s\in E_d$. In other words, $\bm p_u:=\left(p_1(u), \dots, p_d(u)\right)\in P_d^+$ when the individual sign biases $p_1(u), \dots, p_d(u)$ are jointly compatible with a distribution supported on even-parity sign vectors, so that at each magnitude level $u$ the signs can be coordinated to produce a nonnegative product while keeping the absolute values comonotonic. The interpretation of the odd-parity polytope $P_d^-$ is similar to $P_d^+$. Figure~\ref{fig:parity-d3} illustrates the trivariate even- and odd-parity polytopes and their intersection. Specifically, a sign-bias vector $\boldsymbol{p}_u=(p_1(u), p_2(u), p_3(u))$ admits both an even-parity and an odd-parity representation iff it lies in~$P_3^+\cap P_3^-$ in Figure~\ref{fig:parity-d3} (c).

\begin{figure}[ht]
\centering
\begin{subfigure}[b]{0.31\textwidth}
  \centering
  \includegraphics[width=\textwidth]{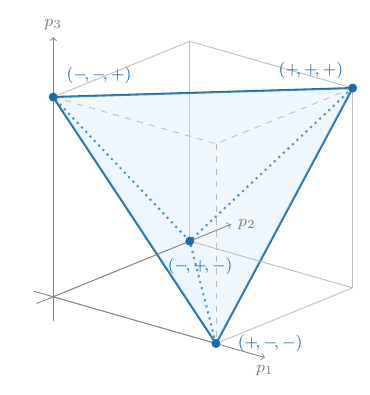}
  \caption{Even-parity polytope $P_3^+$.}
\end{subfigure}
\hfill
\begin{subfigure}[b]{0.31\textwidth}
  \centering
  \includegraphics[width=\textwidth]{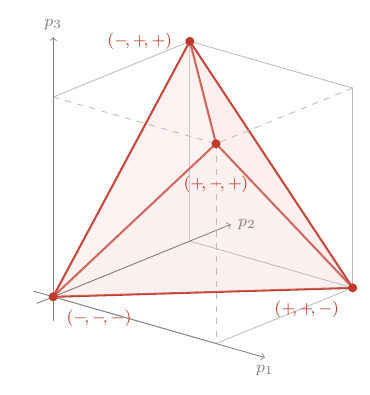}
  \caption{Odd-parity polytope $P_3^-$.}
\end{subfigure}
\hfill
\begin{subfigure}[b]{0.31\textwidth}
  \centering
  \includegraphics[width=\textwidth]{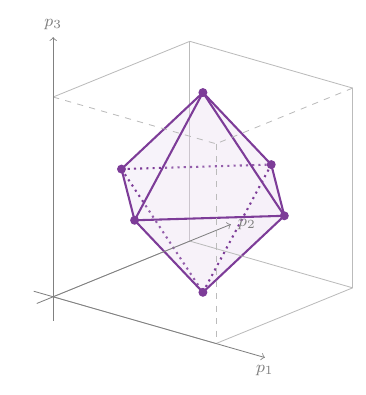}
  \caption{Intersection $P_3^+\cap P_3^-$.}
\end{subfigure}
\caption{Parity polytopes in dimension $d=3$, shown inside the unit cube $[0, 1]^3$. The even-parity polytope $P_3^+=\mathrm{conv}\{(1, 1, 1), (1, 0, 0), (0, 1, 0), (0, 0, 1)\}$ (left, blue) and the odd-parity polytope $P_3^-=\mathrm{conv}\{(1, 1, 0), (1, 0, 1), (0, 1, 1), (0, 0, 0)\}$ (center, red) are congruent regular tetrahedra whose eight vertices together are exactly $\{0, 1\}^3$. In the left and center panels, each tetrahedron vertex is labelled by its sign pattern $\boldsymbol{s}$. Their intersection $P_3^+\cap P_3^-$ (right, purple) is a regular octahedron.}
\label{fig:parity-d3}
\end{figure}

\subsection{Parity feasibility}\label{subsec:parity-feasible}

The equivalences that will be required in Theorem~\ref{thm:attainability} and \ref{thm:lower-attainability} below require that the pointwise feasibility condition $\bm p_u\in P_d^+$ and $\bm p_u\in P_d^-$ for a.e.\ $u$ imply the existence of a Borel-measurable family of weights $u\mapsto w_{\bm s}(u)$ selecting feasible sign-pattern probabilities on the polytopes, respectively. The weights $w_{\bm s}(u)$ represent the probability assigned to each even-parity or odd-parity sign vector $\bm s$ at level $u$, and the constraints on $w_{\bm s}$ ensure that the marginal sign biases are matched. Intuitively, these weights randomize over parity-compatible sign patterns while matching, in each coordinate, the marginal sign bias $p_i(u)$ required at magnitude level $u$. The proof of the following result is in Appendix \ref{app:proofs}.

\begin{lemma}\label{lem:measurable-selection}
Let $\bm p_u:(0, 1)\to[0, 1]^d$ be Borel-measurable, and write $\bm p_u=(p_1(u), \dots, p_d(u))$.
\begin{enumerate}[(1)]
\item If $\bm p_u\in P_d^+$ for a.e.\ $u\in(0, 1)$, then there exist Borel-measurable functions $w_{\bm s}:(0, 1)\to\R$, indexed by $\bm s\in E_d$, such that, for a.e.\ $u$,
$$w_{\bm s}(u)\ge 0, \quad \sum_{\bm s\in E_d} w_{\bm s}(u)=1, \quad \sum_{\bm s\in E_d, \, s_i=+1} w_{\bm s}(u)=p_i(u)\quad\text{for all } i\in[d].$$
\item\label{it:meas-odd} If $\bm p_u\in P_d^-$ for a.e.\ $u\in(0, 1)$, then there exist Borel-measurable functions $w_{\bm s}:(0, 1)\to\R$, indexed by $\bm s\in O_d$, such that, for a.e.\ $u$,
$$w_{\bm s}(u)\ge 0, \quad \sum_{\bm s\in O_d} w_{\bm s}(u)=1, \quad \sum_{\bm s\in O_d, \, s_i=+1} w_{\bm s}(u)=p_i(u)\quad\text{for all } i\in[d].$$
\end{enumerate}
\end{lemma}

In dimension 3, the parity sets $E_3$ in \eqref{eq: travariate even sign set} and $O_3$ in \eqref{eq: travariate odd sign set} each contain four sign patterns, whose associated vertices form the simplices $P_3^+$ and $P_3^-$. Consequently, every point in $P_3^+$ or $P_3^-$ has unique barycentric coordinates, so the sign-pattern probabilities $w_{\bm s}(u)$ from Lemma~\ref{lem:measurable-selection} are uniquely determined at each feasible level~$u$. Therefore, explicit universal upper bounds and the corresponding copulas to obtain them may be derived; we discuss these in Section~\ref{sec:d3}. However, for higher-order mixed moments ($d\ge 4$), the uniqueness does not hold. For example, when $d=4$, the parity sets are
$$\begin{aligned}
    E_4=\{&(+, +, +, +), (+, +, -, -), (+, -, +, -), (+, -, -, +), \\
    &(-, +, +, -), (-, +, -, +), (-, -, +, +), (-, -, -, -)\},
\end{aligned}
$$
and 
$$\begin{aligned}
    O_4=\{&(+, +, +, -), (+, +, -, +), (+, -, +, +), (+, -, -, -), \\
    &(-, +, +, +), (-, +, -, -), (-, -, +, -), (-, -, -, +)\}.
\end{aligned}
$$
Then the constructed polytopes $P_4^+$ and $P_4^-$ may be decomposed into $N=8$ simplices. At each magnitude level $u$, the sign weights $\{w_{\bm s}(u)\}_{\bm s\in E_4}$ must satisfy five linear constraints (normalization and four marginal conditions), leaving three degrees of freedom. The feasible set at each level~$u$, when nonempty, is therefore a three-dimensional polytope, in contrast with $d=3$ where the weights are uniquely determined. It follows that the maximizing and minimizing copulas are not unique in general when $d\geq 4$.

\subsection{Attainability of the maximizer and minimizer}\label{subsec:optimizer}
The following theorem provides the main characterization of attainability for the universal upper bound and an explicit construction of a maximizer.
\begin{theorem}\label{thm:attainability}
Let $U\eqiid V\sim \U[0, 1]$. Under Assumption \ref{ass:regularity}, the following are equivalent:
\begin{enumerate}[(1)]
\item\label{it:attain} There exists a random vector $(X_1, \dots, X_d)$ with $X_i\sim F_i$ for $i\in [d]$ such that
\begin{equation}\label{eq:attain-eq}
\E\left[\prod_{i=1}^d X_i\right] = \E\left[\prod_{i=1}^d G_i^{-1}(U)\right].
\end{equation}

\item\label{it:feasible} For a.e.\ $u\in(0, 1)$, the sign-bias vector
$\bm p_u := \left(p_1(u), \dots, p_d(u)\right)$
belongs to the even-parity polytope $P_d^+$. 
\end{enumerate}
When these conditions hold, the universal upper bound in \eqref{eq:upper-bound} is sharp.
Moreover, a maximizer can be constructed explicitly. Fix an ordering $E_d=\{\bm s^{(1)}, \dots, \bm s^{(m)}\}$ with $m=2^{d-1}$, and define cumulative sums $W_0(u):=0$ and $W_k(u):=\sum_{j=1}^k w_{\bm s^{(j)}}^\ast(u)$ for $k=1, \dots, m$. Let $Y_i:=G_i^{-1}(U)$, define $\bm S:=(S_1, \dots, S_d)$ by
\begin{equation}\label{eq:maximizer-construction}
\bm S:=\bm s^{(k)}\text{ when }W_{k-1}(U)<V\le W_k(U),
\end{equation}
and set $X_i:=S_iY_i$.
Then $\E[\prod_{i=1}^{d}X_i]=M_d^\times$.
\end{theorem}

The construction in \eqref{eq:maximizer-construction} decomposes each factor as $X_i = S_i Y_i$, separating magnitude from sign. To maximize the expected product, the magnitudes $Y_i = G_i^{-1}(U)$ must be comonotonic, while the sign vector $\bm S$ must be coordinated so that $\prod_{i=1}^d S_i = +1$ a.s., ensuring the product is nonnegative. The even-parity feasibility condition $\bm p_u \in P_d^+$ is what allows both requirements to be met simultaneously while respecting the marginal sign biases.

Using Lemma~\ref{lem:product-bounds} and (2) of Lemma~\ref{lem:measurable-selection}, the following theorem is the odd-parity counterpart of Theorem~\ref{thm:attainability}.

\begin{theorem}\label{thm:lower-attainability}
Let $U\eqiid V\sim \U[0, 1]$. Under Assumption~\ref{ass:regularity}, the following are equivalent:
\begin{enumerate}[(1)]
\item\label{it:lower-attain} There exists a random vector $(X_1, \dots, X_d)$ with $X_i\sim F_i$ for $i\in [d]$  such that
\begin{equation}\label{eq:lower-attain-eq}
\E\left[\prod_{i=1}^d X_i\right] = -\E\left[\prod_{i=1}^d G_i^{-1}(U)\right].
\end{equation}

\item\label{it:lower-feasible} For a.e.\ $u\in(0, 1)$, the sign-bias vector
$\bm p_u := \left(p_1(u), \dots, p_d(u)\right)$
belongs to the odd-parity polytope $P_d^-$. 
\end{enumerate}
When these conditions hold, the universal lower bound in \eqref{eq:lower-bound} is sharp.
Moreover, we construct a minimizer as follows.  Fix an ordering $O_d=\{\bm s^{(1)}, \dots, \bm s^{(m)}\}$ with $m=2^{d-1}$, and define cumulative sums $W_0(u):=0$ and $W_k(u):=\sum_{j=1}^k w_{\bm s^{(j)}}^*(u)$ for $k=1, \dots, m$. Let $Y_i:=G^{-1}_i(U)$ and $\bm S:=(S_1, \dots, S_d)$ with $S_i:=\mathrm{sign}(X_i)\in\{-1, +1\}$. Define $X_i:=S_i Y_i$ with
\begin{equation}\label{eq:minimizer-construction}
\bm S:=\bm s^{(k)}\text{ when }W_{k-1}(U)<V\le W_k(U).
\end{equation}
Then $\E[\prod_{i=1}^{d}X_i]=m_d^\times$.
\end{theorem}

The constructions in Theorems~\ref{thm:attainability} and~\ref{thm:lower-attainability} mix over all $2^{d-1}$ sign patterns of the relevant parity. The following proposition shows that this can always be reduced to a sparse mixture.

\begin{proposition}\label{prop:parity-sparsity}
Fix $u\in(0, 1)$ and suppose that $\bm p_u$ belongs to $P_d^+$ or $P_d^-$. Then there exist at most $d+1$ sign vectors $\bm s^{(1)}, \dots, \bm s^{(r)}$ from $E_d$ or $O_d$, respectively, and weights $\lambda_1, \dots, \lambda_r\ge 0$ with $\sum_{j=1}^r\lambda_j=1$ such that
$$\bm p_u=\sum_{j=1}^r \lambda_j v\left(\bm s^{(j)}\right).$$
Equivalently, at each feasible level $u$, one may realize the sign distribution using at most $d+1$ parity-consistent patterns.
\end{proposition}

\begin{proof}
This is a direct application of Carathéodory's theorem, since $P_d^+$ and $P_d^-$ are polytopes in $\R^d$.
\end{proof}

\begin{remark}
The measurable selection guaranteed by Lemma~\ref{lem:measurable-selection} is used only to establish the existence of a Borel-measurable weight function $w^\ast_{\bm s}$; the constructions in Theorems~\ref{thm:attainability} and~\ref{thm:lower-attainability} are well-defined as soon as such a $w^\ast_{\bm s}$ exists. In the trivariate case of Section~\ref{sec:d3}, we do not require this lemma: the polytopes $P_3^+$ and $P_3^-$ are simplices, so the weights are given in closed form. In the higher-dimensional case of Section~\ref{sec:recursive-parity}, we will construct extremal couplings recursively from the trivariate ones, again without invoking the general selector. 
\end{remark}

Theorems~\ref{thm:attainability} and~\ref{thm:lower-attainability} characterize simultaneous sharpness by the condition $\bm p_u\in P_d^+\cap P_d^-$ for a.e.\ $u$. The following corollary gives a marginal-by-marginal sufficient condition that avoids checking this vector condition directly.

\begin{corollary}\label{cor:balanced-marginals}
Under Assumption~\ref{ass:regularity}, suppose that, for every $i\in[d]$,
$$\frac{1}{d-1}f_i(-y)\le f_i(y)\le(d-1)f_i(-y)$$
for a.e.\ $y>0$. Then both universal bounds in \eqref{eq:upper-bound} and \eqref{eq:lower-bound} are sharp.
\end{corollary}

For elliptical location-scale marginals $X_i=\mu_i+\sigma_i Z_i$, suppose that $Z_i$ has a positive symmetric density $h_i$ on $\R$ and that $\log h_i$ is Lipschitz continuous with constant $L_i<\infty$. Corollary~\ref{cor:balanced-marginals} applies whenever
$$2L_i\frac{|\mu_i|}{\sigma_i}\le\log(d-1), \qquad i\in[d].$$
If $Z_i$ has unit variance, then $\mu_i$ and $\sigma_i^2$ are the mean and variance of $X_i$. Thus, a smaller absolute mean or a larger variance keeps the two density branches closer and the sign bias nearer $1/2$. Since the marginal sign-bias vector $\bm p_u$ is then closer to the center of both parity polytopes, it remains in $P_d^+\cap P_d^-$ for a wider range of parameter values.

\section{Identically distributed marginal distributions}\label{sec:iid}
In this section, we assume $X_i\sim F$ with density $f$ and $|X_i|\sim G$. Then
$$p(u)=\frac{f(G^{-1}(u))}{f(G^{-1}(u))+f(-G^{-1}(u))}.$$

Restricting the parity polytopes to the diagonal $p_1=\cdots=p_d=p$ reduces the feasibility conditions to scalar constraints that depend on the parity of $d$. The following proposition translates this diagonal geometry into sharpness criteria for the product expectations.

\begin{proposition}\label{cor:iid}
Let $X_i\sim F$ and $p(u)$ be the common sign-bias function. Under Assumption~\ref{ass:regularity}, $M_d^{\times}=\E\left[G^{-1}(U)^d\right]$ and $m_d^{\times}=-\E\left[G^{-1}(U)^d\right]$ are attainable iff
$(p(u), \dots, p(u))\in P_d^+$ and $(p(u), \dots, p(u))\in P_d^-$ for a.e.\ $u$, respectively. In particular,
\begin{enumerate}[(1)]
  \item if $d$ is even, $M_d^{\times}$ in \eqref{eq:upper-bound}  is always sharp; $m_d^{\times}$ in \eqref{eq:lower-bound} is sharp iff $\frac{1}{d}\le p(u)\le \frac{d-1}{d}$ for a.e.\ $u$;
  \item if $d$ is odd, for a.e\ $u$, $M_d^{\times}$ in \eqref{eq:upper-bound} is sharp iff $p(u)\ge \frac{1}{d}$; $m_d^{\times}$ in \eqref{eq:lower-bound} is sharp iff $p(u)\le \frac{d-1}{d}$. 
\end{enumerate}
In either case, both universal bounds are sharp iff
$$\frac{1}{d-1}f(-y)\le f(y)\le(d-1)f(-y)$$
for a.e.\ $y>0$.
\end{proposition}

When $d=2$, the odd-parity constraint forces $p(u)=1/2$, which by \eqref{eq:sign-bias} requires $f(y)=f(-y)$. This implies that $F$ must be symmetric about zero. As $d$ grows, the feasibility band $[1/d, (d-1)/d]$ widens for even $d$ because the required sign reversal can be distributed across more coordinates. Therefore, $m_d^{\times}$ tolerates progressively more asymmetry.

\begin{corollary}\label{cor:skew-obstruction}
Let $d\ge 2$ and $a>0$, and let $X_i\sim F$ such that $F$ is supported on $[-a, \infty)$ and has unbounded positive support for $i\in[d]$. Under Assumption~\ref{ass:regularity}, $m_d^{\times}$ in \eqref{eq:lower-bound} is not sharp.
\end{corollary}

\begin{proof}
For $y>a$, the value $-y<-a$ lies outside the support of $F$, so the conditional sign of $X$ given $|X|=y$ is necessarily positive; hence $p(u)=1$ for a.e.\ $u\in(G(a), 1)$, which is a set of positive measure since $F$ has unbounded positive support. Since every vertex of $P_d^-$ has coordinate sum at most $d-1$, the point $\mathbf{1}_d$ with coordinate sum $d$ does not belong to $P_d^-$. By Theorem~\ref{thm:lower-attainability}, $m_d^{\times}$ is not sharp.
\end{proof}

By a symmetric argument, if $F$ is supported on $(-\infty, a]$ for some $a>0$ with unbounded negative support, then $p(u)=0$ on a set of positive measure. When $d$ is odd, this violates the condition $p(u)\ge 1/d$ of Proposition~\ref{cor:iid}, so $M_d^{\times}$ is also not sharp. Thus, for both $M_d^{\times}$ and $m_d^{\times}$ to be simultaneously sharp with $X_i\sim F$, $F$ must either be supported on all of $\R$ or have bounded support.

When $d=4$, Proposition~\ref{cor:iid} specializes as follows.

\begin{corollary}\label{prop:d4-iid}
Let $X_i\sim F$. Under Assumption~\ref{ass:regularity},
$M_4^{\times}=\E[G^{-1}(U)^4]$ is always attainable;
$m_4^{\times}=-\E[G^{-1}(U)^4]$ is attainable iff
$\frac{1}{4}\le p(u)\le\frac{3}{4}\text{ for a.e.\ }u\in(0, 1).$
\end{corollary}

\section{Sharp bounds for the trivariate product}\label{sec:d3}
We now specialize the general theory of Section~\ref{sec:sharp-max} to the trivariate case, which is the lowest dimension in which the sign-recombination mechanism is nontrivial. For $d=2$, the product is supermodular, so the optimizers are the classical comonotonic and countermonotonic couplings. For $d=3$, the parity sets $E_3$ in \eqref{eq: travariate even sign set} and $O_3$ in \eqref{eq: travariate odd sign set} each contain exactly four sign vectors, so the polytopes $P_3^+$ and $P_3^-$ each have four vertices in $\R^3$ and are therefore simplices. Any point in a simplex has unique barycentric coordinates, so the sign-pattern probabilities $w_{\bm s}(u)$ from Lemma~\ref{lem:measurable-selection} are uniquely determined at each feasible level~$u$. This uniqueness makes the trivariate case the natural setting for explicit formulas and for the illustrated examples in Appendix~\ref{sec:examples}.

The inequalities below are the $d=3$ specialization of the forbidden-set description in \eqref{eq:parity-forbidden-set}; here they provide explicit attainability conditions and the sign weights used to construct the extremal copulas.

A vector $\bm p_u=(p_1(u), p_2(u), p_3(u))\in[0, 1]^3$ belongs to $P_3^+$ and $P_3^-$ iff the weights
\begin{equation}\label{eq:d3-even-weights}
\begin{aligned}
w_{(+, +, +)}(u)&:=\frac{p_1(u)+p_2(u)+p_3(u)-1}{2}, \quad &w_{(+, -, -)}(u)&:=\frac{1+p_1(u)-p_2(u)-p_3(u)}{2}, \\
w_{(-, +, -)}(u)&:=\frac{1-p_1(u)+p_2(u)-p_3(u)}{2}, \quad &w_{(-, -, +)}(u)&:=\frac{1-p_1(u)-p_2(u)+p_3(u)}{2}
\end{aligned}
\end{equation} and \begin{equation}\label{eq:d3-odd-weights}
\begin{aligned}
w_{(+, +, -)}(u)&:=\frac{p_1(u)+p_2(u)-p_3(u)}{2}, \quad &w_{(+, -, +)}(u)&:=\frac{p_1(u)-p_2(u)+p_3(u)}{2}, \\
w_{(-, +, +)}(u)&:=\frac{-p_1(u)+p_2(u)+p_3(u)}{2}, \quad &w_{(-, -, -)}(u)&:=\frac{2-p_1(u)-p_2(u)-p_3(u)}{2}
\end{aligned}
\end{equation}
are all nonnegative, respectively. Equivalently, for a.e.\ $u$, one must have the inequalities
$$p_1(u)+p_2(u)+p_3(u)\ge 1, \qquad p_1(u)\ge p_2(u)+p_3(u)-1,$$
$$p_2(u)\ge p_1(u)+p_3(u)-1, \qquad p_3(u)\ge p_1(u)+p_2(u)-1,$$
for the universal upper bound and $$p_1(u)+p_2(u)+p_3(u)\le 2, \qquad p_1(u)\le p_2(u)+p_3(u),$$
$$p_2(u)\le p_1(u)+p_3(u), \qquad p_3(u)\le p_1(u)+p_2(u),$$ for the universal lower bound. When these hold at a.e.\ level $u$, Theorem~\ref{thm:attainability} (resp., \ref{thm:lower-attainability}) guarantees that the universal upper (resp., lower) bound of $M_d^{\times}$ (resp., $m_d^{\times}$) is attained. Moreover, the maximizing (resp., minimizing) coupling assigns sign vector $\bm s\in E_3$ (resp., $\bm s\in O_3$) with conditional probability $w_{\bm s}(u)$ given $U=u$ and sets $X_i=S_i Y_i$.

\subsection{Maximizing and minimizing copulas}\label{subsec:d3-copula}

We now derive the copulas associated with the extremal couplings of Theorems~\ref{thm:attainability} and~\ref{thm:lower-attainability}. Throughout this section, Assumption~\ref{ass:regularity} continues to hold.

\begin{proposition}\label{prop:d3-copula}
Let Assumption~\ref{ass:regularity} hold for $F_1, F_2, F_3$. Let $u_{0, i}= F_i(0)$, $|X_i|\sim G_i$, $U\eqiid V \sim \U[0, 1]$, $I = \id_{\{U > u_{0, 1}\}}$ and $W = G_1\left(|F_1^{-1}(U)|\right)$ with $w= G_1\left(|F_1^{-1}(u)|\right)$ for $u\in(0, 1)$. Define the Borel-measurable branch maps $h_i^+(w) = F_i\left(G_i^{-1}(w)\right)$ and $h_i^-(w) = F_i\left(-G_i^{-1}(w)\right)$, and let $J = \id_{\{V \le s(U)\}}$.

If $\left(p_1(w), p_2(w), p_3(w)\right)\in P_3^+$ for a.e.\ $w$, then the universal upper bound $M_3^{\times}=\E\left[\prod_{i=1}^3 G_i^{-1}(U)\right]$ is attained by $X_i = F_i^{-1}(U_i)$, where
\begin{equation}\label{eq:d3-maximizer-copula}
\begin{aligned}
U_1 &:= U, \\
U_2 &:= J h_2^+(W) + (1-J) h_2^-(W), \\
U_3 &:= I\left[J h_3^+(W) + (1-J) h_3^-(W)\right] + (1-I)\left[J h_3^-(W) + (1-J) h_3^+(W)\right],
\end{aligned}
\end{equation}with mixing function
\begin{equation}\label{eq:d3-maximizer-mixing}
s(u) = \begin{cases}
\dfrac{p_1(w)+p_2(w)+p_3(w)-1}{2 p_1(w)}, & u > u_{0, 1}, \\[8pt]
\dfrac{1-p_1(w)+p_2(w)-p_3(w)}{2 \left(1-p_1(w)\right)}, & u \le u_{0, 1}.
\end{cases}
\end{equation}

If $(p_1(w), p_2(w), p_3(w))\in P_3^-$ for a.e.\ $w$, then the universal lower bound $m_3^{\times}=-\E[\prod_{i=1}^3 G_i^{-1}(U)]$ is attained by $X_i = F_i^{-1}(U_i)$, where
\begin{equation}\label{eq:d3-minimizer-copula}
\begin{aligned}
U_1 &:= U, \\
U_2 &:= J h_2^+(W) + (1-J) h_2^-(W), \\
U_3 &:= I\left[J h_3^-(W) + (1-J) h_3^+(W)\right] + (1-I)\left[J h_3^+(W) + (1-J) h_3^-(W)\right],
\end{aligned}
\end{equation}
with mixing function
\begin{equation}\label{eq:d3-minimizer-mixing}
s(u) = \begin{cases}
\dfrac{p_1(w)+p_2(w)-p_3(w)}{2 p_1(w)}, & u > u_{0, 1}, \\[8pt]
\dfrac{-p_1(w)+p_2(w)+p_3(w)}{2 \left(1-p_1(w)\right)}, & u \le u_{0, 1}.
\end{cases}
\end{equation}
The displayed fractions are used only on the corresponding active branch. On null or inactive branches, $s(u)$ may be chosen arbitrarily in $[0, 1]$.
\end{proposition}

\begin{remark}
    The branch maps $h_i^+(w)$ and $h_i^-(w)$ direct the magnitude quantile $w$ to the uniform scale of coordinate~$i$ on the positive and negative branches, respectively. If one defines the sign-flip map $\tau_i(u)=F_i\left(-F_i^{-1}(u)\right)$ on levels where both branches exist, then $h_i^-(w)=\tau_i\left(h_i^+(w)\right)$. This identity is only notational; the construction uses the branch maps $h_i^\pm$ directly and does not require any smoothness of $\tau_i$. The role of these maps is to absorb the heterogeneity of $G_1$, $G_2$ and $G_3$ into coordinate-specific transformations, so that the sign structure may be governed by a single Bernoulli variable whose bias depends on the sign-bias vector $\left(p_1(w), p_2(w), p_3(w)\right)$. The mixing functions~\eqref{eq:d3-maximizer-mixing} and~\eqref{eq:d3-minimizer-mixing} involve division by $p_1(w)$ and $1-p_1(w)$ respectively. These quantities are strictly positive a.e.\ iff $F_1$ assigns positive mass to both $(-\infty, 0)$ and $(0, \infty)$, that is, when $u_{0, 1}=F_1(0)\in(0, 1)$. Under Assumption~\ref{ass:regularity}, this holds whenever the connected support of $F_1$ is not contained entirely in $[0, \infty)$ or $(-\infty, 0]$.
\end{remark}

Table~\ref{tab:d3-max-patterns} enumerates the four possible pairs of $(I, J)$. For $I = 1$ and the maximizer, coordinates~$2$ and~$3$ are placed on matching branches ($h^+_i$ or $h^-_i$), while they are assigned to different branches if $I = 0$. In contrast, for the minimizer, the third coordinate's branch assignment is reversed. Specifically, when $I = 1$, coordinates~$2$ and~$3$ are placed on opposite branches; otherwise, they are assigned to the same branch. In all four cases, the sign pattern for the maximizer (resp., minimizer) belongs to the even-parity (resp., odd-parity) set $E_3$ (resp., $O_3$), so $X_1 X_2 X_3 \ge 0$ (resp., $X_1 X_2 X_3 \le 0$) a.s., and the absolute values are comonotonic since all three are driven by a common magnitude variable $W$.

\begin{table}[ht]
\centering
\begin{adjustbox}{max width=\textwidth}
\begin{tabular}{ccccccllc}
\toprule
\multirow{2}{*}{$I$} & \multirow{2}{*}{$J$} & \multicolumn{2}{c}{Sign patterns} & \multicolumn{2}{c}{$(U_1, U_2, U_3)$} & \multirow{2}{*}{cond.\ prob.} & \multirow{2}{*}{active when} & \multirow{2}{*}{Fig.~\ref{fig:linear-maximizer}-\ref{fig:linear-minimizer}} \\
\cmidrule(r){3-4} \cmidrule(l){5-6}
& & $\bm s\in E_3$ & $\bm s\in O_3$ & $M_3^{\times}$ & $m_3^{\times}$ &  &  & \\
\midrule
$1$ & $1$ & $(+, +, +)$ & $(+, +, -)$ & $(u, h_2^+(w), h_3^+(w))$ & $(u, h_2^+(w), h_3^-(w))$ & $s(u)$   & $u>u_{0, 1}$ & \textcolor{legcol1}{$\blacksquare$} \\
$1$ & $0$ & $(+, -, -)$ & $(+, -, +)$ & $(u, h_2^-(w), h_3^-(w))$ & $(u, h_2^-(w), h_3^+(w))$ & $1-s(u)$ & $u>u_{0, 1}$ & \textcolor{legcol2}{$\blacksquare$} \\
$0$ & $1$ & $(-, +, -)$ & $(-, +, +)$ & $(u, h_2^+(w), h_3^-(w))$ & $(u, h_2^+(w), h_3^+(w))$ & $s(u)$   & $u\le u_{0, 1}$ & \textcolor{legcol4}{$\blacksquare$} \\
$0$ & $0$ & $(-, -, +)$ & $(-, -, -)$ & $(u, h_2^-(w), h_3^+(w))$ & $(u, h_2^-(w), h_3^-(w))$ & $1-s(u)$ & $u\le u_{0, 1}$ & \textcolor{legcol3}{$\blacksquare$} \\
\bottomrule
\end{tabular}
\end{adjustbox}
\caption{Sign patterns, support legs, and conditional probabilities of the trivariate maximizer and minimizer for general marginals. Each row corresponds to one $(I, J)$ pair; $w = G_1(|F_1^{-1}(u)|)$ is the magnitude quantile associated with $u$. The indicator $I = \id_{\{u > u_{0, 1}\}}$ is deterministic given $U = u$, while $J = \id_{\{V \le s(u)\}}$ is an independent Bernoulli variable with conditional probability $P(J=1\mid U=u) = s(u)$ in \eqref{eq:d3-maximizer-mixing} (resp., \eqref{eq:d3-minimizer-mixing}) for the maximizer (resp., minimizer). The last column indicates the colour of the corresponding support leg in Figures~\ref{fig:linear-maximizer} and \ref{fig:linear-minimizer}.}
\label{tab:d3-max-patterns}
\end{table}

When the marginals are identical, $F_1 = F_2 = F_3 =: F$, the maps $h_i^{\pm}$ coincide and $W$ is a two-to-one function of~$U$, with $u$ and $\tau(u)$ mapping to the same magnitude level. On each branch for the maximizer, the map is invertible: for $U > u_0$ one has $h^+(W) = U$ and $h^-(W) = \tau(U)$, while for $U \le u_0$ the roles reverse. The mixing function becomes $s(u) = (3p(w)-1)/(2p(w))$ on $(u_0, 1)$ and $s(u) = 1/2$ on $(0, u_0]$, and the copula~\eqref{eq:d3-maximizer-copula} places $U_2, U_3 \in \{U, \tau(U)\}$.   

In the case of the minimizer, the mixing function becomes $s(u) = 1/2$ on $(u_0, 1)$ and $s(u) = p(w)/(2(1-p(w)))$ on $(0, u_0]$. When $I = 1$, the Bernoulli variable is symmetric, so the single negative sign is placed on coordinate~$2$ or~$3$ with equal probability. When $I = 0$, the two-positive pattern $(-, +, +)$ carries weight $p(w)/(2(1-p(w)))$, while the all-negative pattern $(-, -, -)$ carries the remaining weight. Thus $s(u)\in[0, 1]$ on the negative branch iff $p(w)\le 2/3$, recovering the odd-parity threshold of Proposition~\ref{cor:iid} for $d = 3$.

\section{A recursive parity construction}\label{sec:recursive-parity}
For $d\ge 4$, the situation is fundamentally different from the trivariate case. For a given sign-bias vector $\bm p_u$, the set of compatible sign weights is no longer a single point, and its dimension generally grows with $d$. As a result, the sign-weight function $w_{\bm s}(u)$ is not uniquely determined by $\bm p_u$. Indeed, any convex decomposition of $\bm p_u$ into vertices of $P_d^+$ or $P_d^-$ produces a valid collection of sign weights, and different decompositions lead to different copulas, all of which attain the same extremal expected product. By contrast, this non-uniqueness does not occur in dimension three: since $P_3^+$ is a simplex, each sign-bias vector admits a unique decomposition, and hence the associated extremal copula is uniquely determined.

In this section, we show how to construct a $d$-dimensional extremal copula recursively from the explicit trivariate construction developed in Section~\ref{sec:d3}. The key idea is that the constructions of Section~\ref{sec:d3} possess a recursive structure that extends naturally to higher dimensions. More precisely, membership in the parity polytope can be decomposed by conditioning on a single pivot coordinate, which provides the basis for the recursive construction.

\begin{proposition}\label{prop:recursive-parity}
Let $d\ge 3$ and $\bm p=(p_1, \dots, p_d)\in[0, 1]^d$. If $p_1\in(0, 1)$, then $\bm p\in P_d^+$ iff there exist $\bm q^+\in P_{d-1}^+$ and $\bm q^-\in P_{d-1}^-$ such that
\begin{equation}\label{eq:recursive-decomposition}
p_j=p_1 q_j^++(1-p_1) q_j^-, \qquad j=2, \dots, d.
\end{equation}
If $p_1=1$, then $\bm p\in P_d^+$ iff $(p_2, \dots, p_d)\in P_{d-1}^+$. If $p_1=0$, then $\bm p\in P_d^+$ iff $(p_2, \dots, p_d)\in P_{d-1}^-$.

The odd-parity analogue holds with the roles of $P_{d-1}^+$ and $P_{d-1}^-$ interchanged: $\bm p\in P_d^-$ iff there exist $\bm q^+\in P_{d-1}^-$ and $\bm q^-\in P_{d-1}^+$ satisfying~\eqref{eq:recursive-decomposition}.
\end{proposition}

The decomposition extends to arbitrary dimension by induction: the $d$-dimensional parity-polytope membership problem reduces to a pair of $(d-1)$-dimensional problems, one on each branch of the pivot coordinate, down to the base case $d=2$. The choice of pivot is arbitrary; different pivots produce different recursive decompositions that do not affect the feasibility of $\bm p_u$ in the parity polytope but reorganize the conditional structure. When $\bm p_u$ depends measurably on $u$, each conditional vector $\bm q^{\pm}(u)$ inherits measurability from the measurable selection of Lemma~\ref{lem:measurable-selection}, so the recursive construction yields a valid measurable coupling.

Combined with the explicit trivariate copulas in Section~\ref{subsec:d3-copula}, Proposition~\ref{prop:recursive-parity} yields explicit $d$-dimensional extremal copulas whenever the parity-feasibility conditions hold. We describe the maximizer for $d=4$; the minimizer is obtained by interchanging the residual parity on each branch.

Fix coordinate~$1$ as the pivot. Let $U\eqiid V\sim\U[0, 1]$, set $W:=G_1(|F_1^{-1}(U)|)$, and define $I:=\id_{\{U>u_{0, 1}\}}$, so that the sign of $X_1$ is determined by~$I$. Conditioning on the value of~$S_1$ produces two trivariate subproblems on coordinates $(2, 3, 4)$:
\begin{itemize}
\item On the branch $S_1=+1$ (i.e., $I=1$), the residual sign vector $(S_2, S_3, S_4)$ must have even parity. The conditional sign biases are $\bm q^+(u)=(q_2^+(u), q_3^+(u), q_4^+(u))\in P_3^+$, and the trivariate even-parity copula in \eqref{eq:d3-maximizer-copula} applies: coordinate~$2$ is the residual pivot, coordinate~$3$ is randomized conditionally on that pivot, and coordinate~$4$ is fixed by parity.
\item On the branch $S_1=-1$ (i.e., $I=0$), the residual vector must have odd parity. The conditional sign biases are $\bm q^-(u)\in P_3^-$, and the trivariate odd-parity copula in \eqref{eq:d3-minimizer-copula} applies with the same residual pivot and randomization.
\end{itemize}
The conditional biases $\bm q^{\pm}(u)$ must satisfy the compatibility condition~\eqref{eq:recursive-decomposition}; any pair $(\bm q^+(u), \bm q^-(u))$ with $\bm q^+(u)\in P_3^+$, $\bm q^-(u)\in P_3^-$, and $p_j(u)=p_1(u)q_j^+(u)+(1-p_1(u))q_j^-(u)$ for $j=2, 3, 4$ yields a valid extremal copula. Since $P_3^+$ and $P_3^-$ are simplices, the barycentric weights on each branch are uniquely determined by $\bm q^{\pm}(u)$, and the mixing functions $s^{\pm}$ are explicit functions of the conditional biases. Different choices of $(\bm q^+(u), \bm q^-(u))$ satisfying the compatibility condition produce different copulas that all attain the same extremal expected product; the nonuniqueness of the $d$-dimensional maximizing and minimizing copulas is therefore parameterized by this choice.

Let $\bm R^+(w, v;\bm q)$ and $\bm R^-(w, v;\bm q)$ denote the trivariate even- and odd-parity sign selectors determined by \eqref{eq:d3-even-weights} and \eqref{eq:d3-odd-weights}, respectively. Concretely, $v$ selects among the four patterns by partitioning $[0, 1]$ into consecutive intervals whose lengths are their barycentric weights. Thus, $\bm R^\pm$ is categorical rather than Bernoulli and may select all four patterns. The four-dimensional maximizer is then generated by
$$S_1:=2I-1, \qquad (S_2, S_3, S_4):=I\bm R^+(W, V;\bm q^+(W))+(1-I)\bm R^-(W, V;\bm q^-(W)),$$
$$U_1:=U, \qquad U_i:=\id_{\{S_i=+1\}}h_i^+(W)+\id_{\{S_i=-1\}}h_i^-(W), \quad i=2, 3, 4, \qquad X_i:=F_i^{-1}(U_i).$$
For the minimizer, interchange the two residual selectors:
$$(S_2, S_3, S_4):=I\bm R^-(W, V;\bm q^+(W))+(1-I)\bm R^+(W, V;\bm q^-(W)).$$
Thus each branch of the four-dimensional construction has the trivariate form of Section~\ref{subsec:d3-copula}.

\section{Numerical illustration}\label{sec:numerical}
In this section, we numerically illustrate the implications of the sharpness characterizations in \eqref{eq:attain-eq} and \eqref{eq:lower-attain-eq}. When the function $\psi$ in \eqref{sharp bounds} is supermodular, the BRA provides an efficient method to approximate the bounds $m_d^{\psi}$ and $M_d^{\psi}$. Moreover, the BRA is also valid for certain non-supermodular functions, such as the product function; see, e.g., \cite{bernard2023coskewness}. Therefore, we use it in this section to approximate $m_d^{\times}$ and $M_d^{\times}$. All numerical approximations reported in this section are rounded to two decimal places. We first consider marginal distributions satisfying the parity conditions, for which the BRA approximations agree with the universal bounds, and then violate one condition at a time to illustrate the resulting gaps.

\subsection{Examples of sharpness}
We begin with heterogeneous bounded marginals in which the even- and odd-parity conditions of Theorems~\ref{thm:attainability} and~\ref{thm:lower-attainability} hold simultaneously. Thus both universal bounds are attainable.

\begin{example}
For $i=1, 2, 3$, let $X_i$ have density
$$f_i(x) = \frac{1+\theta_i x}{2}, \qquad x\in[-1, 1],$$
where $\theta_i\in(-1, 1)$ are parameters. The sign of $\theta_i$ controls the skewness of $F_i$: positive values produce a linearly increasing density, negative values a linearly decreasing density, and $\theta_i=0$ gives the uniform density on $[-1, 1]$. The mean of $F_i$ is $\theta_i/3$.

One may show that $p_i(y) = (1+\theta_i y)/2$ for $y\in(0, 1)$ and that $|X_i|\sim\mathrm{U}[0, 1]$ regardless of $\theta_i$, so the heterogeneity of the marginals enters exclusively through the sign structure, not the magnitudes. Since $p_i$ is monotone in $y$, all parity constraints are tightest at $y=1$. One may verify that both parity conditions hold simultaneously iff
\begin{equation}\label{eq:linear-both-feasibility}
|\theta_1+\theta_2+\theta_3|\le 1 \quad\text{and}\quad |\theta_i+\theta_j-\theta_k|\le 1 \quad\text{for every permutation } (i, j, k) \text{ of } (1, 2, 3).
\end{equation}
The conditions in~\eqref{eq:linear-both-feasibility} are jointly equivalent to $|\theta_1|+|\theta_2|+|\theta_3|\le 1$; the feasibility region is a cross-polytope with nonempty interior, displayed in Figure~\ref{fig:parity-polytope}. Following Theorems~\ref{thm:attainability} and \ref{thm:lower-attainability}, $M_3^{\times} = \E[U^3] = 1/4$ and $m_3^{\times} = -1/4$. Specifically, these values hold for any feasible $(\theta_1, \theta_2, \theta_3)$; the heterogeneity of the marginals affects only the sign-weight functions $w_{\bm s}(u)$, and therefore the shape of the extremal copulas. Figures~\ref{fig:linear-maximizer} and \ref{fig:linear-minimizer} display the corresponding extremal copula supports when $(\theta_1, \theta_2, \theta_3)=(0.4, 0.2, -0.3)$.

\begin{figure}[H]
\color{black}
\centering
\begin{subfigure}[t]{0.32\textwidth}
  \centering
  \includegraphics[width=\textwidth]{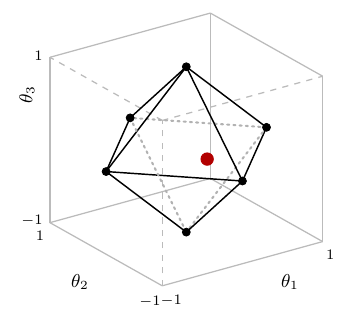}
  \caption{Feasibility polytope.}
  \label{fig:parity-polytope}
\end{subfigure}\hfill
\begin{subfigure}[t]{0.32\textwidth}
  \centering
  \includegraphics[width=\textwidth]{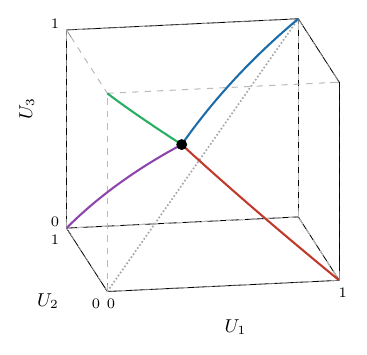}
  \caption{Maximizing copula.}
  \label{fig:linear-maximizer}
\end{subfigure}\hfill
\begin{subfigure}[t]{0.32\textwidth}
  \centering
  \includegraphics[width=\textwidth]{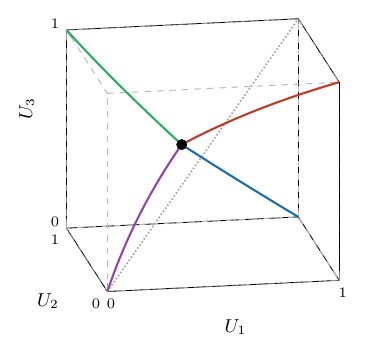}
  \caption{Minimizing copula.}
  \label{fig:linear-minimizer}
\end{subfigure}
\caption{Linear densities with $(\theta_1, \theta_2, \theta_3)=(0.4, 0.2, -0.3)$. (a)~Feasibility region: the cross-polytope $|\theta_1|+|\theta_2|+|\theta_3|\le 1$; the red dot represents the specific numerical example. (b)~Support of the maximizing copula in $[0, 1]^3$; each colour corresponds to one even-parity sign pattern. (c)~Support of the minimizing copula; each colour corresponds to one odd-parity sign pattern. In (b) and (c), the black dot marks the junction $(u_{0, 1}, u_{0, 2}, u_{0, 3})=(0.4, 0.45, 0.575)$, where $u_{0, i}=F_i(0)$ is the probability that $X_i$ is negative. The four support legs in each panel meet at this point because it is the unique level at which all margins change sign simultaneously. Taken together, the maximizing and minimizing legs reach all eight corners of the cube $[0, 1]^3$.}
\label{fig:linear-densities}
\end{figure}

\begin{table}[ht]
    \centering
    \begin{threeparttable}
        \caption{Sharp bounds and BRA approximations for the three-dimensional
        product problem.}
        \label{tab:bra-comparison}

        \renewcommand{\arraystretch}{1.15}

        \begin{tabular*}{\textwidth}
    {@{\extracolsep{\fill}}lcccc@{}}
            \toprule
            & \multicolumn{2}{c}{Sharp bounds}
            & \multicolumn{2}{c}{BRA approximations} \\
            \cmidrule(lr){2-3}
            \cmidrule(lr){4-5}
            $(\theta_1, \theta_2, \theta_3)$
            & $m_3^{\times}=-\E[U^3]$
            & $M_3^{\times}=\E[U^3]$
            & $\widetilde{m}_3^{\times}$
            & $\widetilde{M}_3^{\times}$ \\
            \midrule

            $(0.4, 0.2, -0.3)$
            & $-0.25$
            & $0.25$
            & $-0.25$
            & $0.25$ \\

            \addlinespace

            $(-0.1, -0.2, 0.3)$
            & $-0.25$
            & $0.25$
            & $-0.25$
            & $0.25$ \\

            \addlinespace

            $(0.3, 0.2, 0.1)$
            & $-0.25$
            & $0.25$
            & $-0.25$
            & $0.25$ \\

            \bottomrule
        \end{tabular*}

    \end{threeparttable}
\end{table}

In Table~\ref{tab:bra-comparison}, we compare the sharp bounds with the BRA approximations. The BRA approximations are averages over $1{,}000$ replications with $10^4$ simulations for each marginal distribution. The BRA approximations are close to the exact lower and upper bounds for all parameter configurations considered, numerically confirming that both proposed bounds are sharp in this example.
\end{example}

\subsection{Counterexamples to sharpness}

We now return to the heterogeneous marginals with linear densities and choose parameter vectors for which one or both parity conditions fail. This yields three cases: only the universal lower bound is sharp, only the universal upper bound is sharp, and neither universal bound is sharp.

\begin{table}[htbp]
    \centering
    \begin{threeparttable}
        \caption{Analytic conclusions and BRA approximations for the three-dimensional
        product problem. Dashed entries means the universal bounds are not achieved.
        }
        \label{tab:bra-comparison for counter examples}

        \renewcommand{\arraystretch}{1.15}

        \begin{tabular*}{\textwidth}
    {@{\extracolsep{\fill}}lcccc@{}}
            \toprule
                & \multicolumn{2}{c}{Sharp bounds}
            & \multicolumn{2}{c}{BRA approximations} \\
            \cmidrule(lr){2-3}
            \cmidrule(lr){4-5}
            $(\theta_1, \theta_2, \theta_3)$
            & $m_3^{\times}=-\E[U^3]$
            & $M_3^{\times}=\E[U^3]$
            & $\widetilde{m}_3^{\times}$
            & $\widetilde{M}_3^{\times}$ \\
            \midrule

            $(-0.99, -0.98, -0.97)$
            & $-0.25$
            & --
            & $-0.25$
            & $0.18$ \\

            \addlinespace

            $(0.99, 0.98, 0.97)$
            & --
            & $0.25$
            & $-0.18$
            & $0.25$ \\

            \addlinespace

            $(0.99, -0.99, 0)$
            & --
            & --
            & $-0.24$
            & $0.24$ \\

            \bottomrule
        \end{tabular*}
    \end{threeparttable}
\end{table}

Table~\ref{tab:bra-comparison for counter examples} displays the three possible sharpness patterns when at least one parity condition fails. For $(\theta_1, \theta_2, \theta_3)=(-0.99, -0.98, -0.97)$, the universal lower bound is sharp, and the BRA lower approximation agrees with $-0.25$, whereas the BRA upper approximation is $0.18$, strictly below the universal upper bound of $0.25$. For the reflected parameter vector $(0.99, 0.98, 0.97)$, the opposite pattern holds: the universal upper bound is sharp, while the BRA lower approximation is $-0.18$, strictly above the universal lower bound of $-0.25$. Finally, for $(0.99, -0.99, 0)$, neither universal bound is sharp, and the BRA approximations $-0.24$ and $0.24$ lie strictly inside the universal interval $[-0.25, 0.25]$. These numerical gaps illustrate the loss of sharpness when the corresponding parity condition fails, while the BRA continues to approximate the actual extremal values.

\section{Application to coskewness measurement}\label{sec:coskewness-application}

Mean--variance portfolio optimization and the CAPM describe dependence between assets through covariance and therefore do not capture asymmetric higher-order comovements \citep{markowitz1952portfolio, sharpe1964capital}. The mean--variance decision criterion is obtained as a second-order Taylor approximation of the quadratic expected utility. When one wants to consider criteria with higher-order terms, one must also consider higher-order cross-product terms such as the coskewness. \citet{harvey2000conditional} show that it helps explain cross-sectional variation in expected returns where the traditional CAPM performs poorly. It is therefore a practically important source of risk beyond covariance. Because the feasible range of raw coskewness depends on the marginal distributions, we use the lower and upper dependence-uncertainty bounds to map it onto the fixed interval $[-1, 1]$, yielding a marginally calibrated measure that can be compared across portfolios and over time.

Let $X_i\sim F_i$, $i=1, 2, 3$, denote the returns of three stocks, with means $\mu_i$ and standard deviations $\sigma_i>0$. Their coskewness is
$$S(X_1, X_2, X_3):= \frac{\E[(X_1-\mu_1)(X_2-\mu_2)(X_3-\mu_3)]} {\sigma_1\sigma_2\sigma_3}.$$
When the marginal return distributions are known but their dependence structure is not, the corresponding dependence-uncertainty bounds are
\begin{subequations}
\begin{align}
\underline S&:=\inf_{X_i\sim F_i, \ i=1, 2, 3}S(X_1, X_2, X_3),\label{eq:coskewness-min}\\
\overline S&:=\sup_{X_i\sim F_i, \ i=1, 2, 3}S(X_1, X_2, X_3).\label{eq:coskewness-max}
\end{align}
\end{subequations}
These quantities have a direct portfolio interpretation. If $X_{\bm a}:=\sum_{i=1}^3a_iX_i$ is a three-asset portfolio return, then the term involving all three assets in the third central moment of the portfolio return is
\begin{equation}\label{eq:portfolio-coskewness-term}
6a_1a_2a_3\E[(X_1-\mu_1)(X_2-\mu_2)(X_3-\mu_3)]
=6a_1a_2a_3\sigma_1\sigma_2\sigma_3S(X_1, X_2, X_3).
\end{equation}
For a portfolio without short selling, negative coskewness lowers $\mathbb{E}[(X_{\bm a}-\mathbb{E}[X_{\bm a}])^3]$ and therefore represents an adverse dependence contribution. Thus, two portfolios with the same marginal return distributions can have materially different higher-order risk because their dependence structures differ. Since the marginal distributions can generally be estimated more reliably than the joint dependence, $[\underline S, \overline S]$ provides the marginally admissible range of coskewness values. $\underline S$ yields the most adverse value of the contribution in~\eqref{eq:portfolio-coskewness-term}, whereas $\overline S$ yields the most favorable. These bounds therefore provide the marginally calibrated endpoints used below to locate the observed coskewness within its feasible range and map it onto a common scale.

Define the standardized returns by
$$Y_i:=\frac{X_i-\mu_i}{\sigma_i}\sim H_i, \qquad H_i^{-1}(u)=\frac{F_i^{-1}(u)-\mu_i}{\sigma_i}, \qquad i=1, 2, 3.$$
Then $S(X_1, X_2, X_3)=\E[Y_1Y_2Y_3]$, so Problems~\eqref{eq:coskewness-min} and~\eqref{eq:coskewness-max} are precisely the product-expectation problems for the standardized marginals $H_i$. Under Assumption~\ref{ass:regularity}, let $G_i$ be the df of $|Y_i|$, let $f_i$ and $h_i$ be the densities of $F_i$ and $H_i$, respectively, and define the standardized sign-bias function
$$
p_i(u):=
\frac{h_i(G_i^{-1}(u))}{h_i(G_i^{-1}(u))+h_i(-G_i^{-1}(u))}
=\frac{f_i(\mu_i+\sigma_iG_i^{-1}(u))}
{f_i(\mu_i+\sigma_iG_i^{-1}(u))+f_i(\mu_i-\sigma_iG_i^{-1}(u))}.
$$
Thus, $p_i(u)$ is the conditional probability that a standardized deviation of magnitude $G_i^{-1}(u)$ lies above, rather than below, the mean of the $i^{\mathrm{th}}$ return. The even- and odd-parity conditions applied to these sign-bias functions determine whether the universal upper and lower coskewness bounds are sharp. When either condition fails, the corresponding universal bound is not sharp, and the true extremal coskewness can instead be approximated using the BRA.

Coskewness depends on the marginal distributions, so its magnitude cannot be compared directly across different sets of marginals. We therefore use its lower and upper bounds to map its feasible range onto the fixed interval $[-1, 1]$ and define the following standardized measure of comovement.
\begin{definition}[Standardized coskewness]\label{standardized coskewness}
	Let $X_i\sim F_{i}$, $i=1, 2, 3$. Assume that $-\infty<\underline{S}<\overline{S}<\infty$. The standardized coskewness of $X_1$, $X_2$ and $X_3$ denoted by $SS(X_1, X_2, X_3)$ is defined as
	\begin{equation} 
		SS(X_1, X_2, X_3)=\frac{2}{S(X_1, X_2, X_3)-(\underline{S}+\overline{S})}{\overline{S}-\underline{S}}.
		\label{equation for stand coskewness} 
	\end{equation}
    If $p_i(u)$ satisfy the conditions in \eqref{eq:d3-even-weights} and \eqref{eq:d3-odd-weights} for a.e.\ $u\in(0, 1)$ simultaneously, then $\underline{S}=-B_S$ and $\overline{S}=B_S$, where $B_S:=\E[\prod_{i=1}^3 G_i^{-1}(U)]$, and $$SS(X_1, X_2, X_3)=\frac{S(X_1, X_2, X_3)}{B_S}.$$
\end{definition}

Using the same data for Figure~\ref{fig:rolling-three-stock-coskewness}, Figure~\ref{fig:rolling-three-stock-stand-coskewness} shows the empirical standardized coskewness with a 60-trading-day rolling window from January 2005 to December 2025. Unlike the raw coskewness in Figure~\ref{fig:rolling-three-stock-coskewness}, whose feasible range changes with the marginal distributions, each standardized value lies in the fixed interval $[-1, 1]$. By construction, $SS(X_1, X_2, X_3)$ rescales the normalized distance $(S(X_1, X_2, X_3)-\underline{S})/(\overline{S}-\underline{S})$ from the lower bound onto this fixed interval, so $-1$ and $1$ correspond to the lower and upper bounds, respectively, while zero corresponds to the midpoint of the feasible interval. The transformation therefore retains a signed interpretation: negative values place the observed coskewness closer to the lower bound, whereas positive values place it closer to the upper bound. Because the bounds are recomputed for each 60-day window, standardized coskewness provides a consistent measure of comovement, allowing its level and trend to be compared across time.
\begin{figure}[H]
  \centering
  \includegraphics[width=\textwidth]{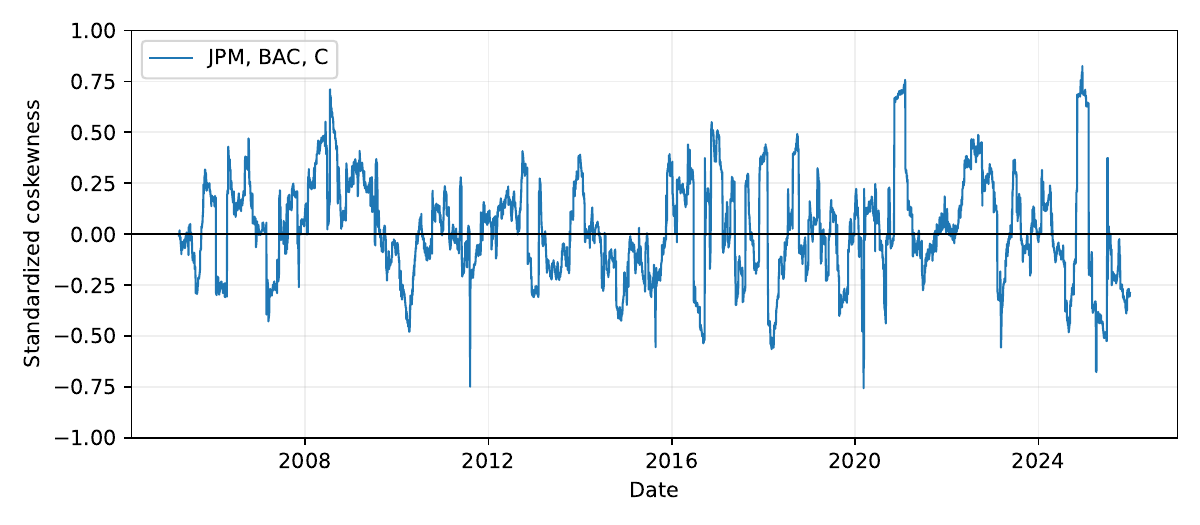}
  \caption{Rolling standardized coskewness of daily returns on three financial institutions, JPM, BAC, and C, computed over 60-trading-day windows. The standardized coskewness is computed as in Equation~\eqref{equation for stand coskewness}, where $X_i$ denotes
the companies' daily returns and risk bounds ($\underline{S}$ and $\overline{S}$) are computed by the BRA algorithm. Returns are calculated from adjusted closing prices from Yahoo Finance over January 2005--December 2025.}
  \label{fig:rolling-three-stock-stand-coskewness}
\end{figure}

\section{Conclusion}

In this paper, we establish necessary and sufficient conditions for the universal upper and lower bounds on $\E[X_1\cdots X_d]$ to be sharp when only the marginal distributions are specified. The analysis relies on decomposing the product's magnitudes and signs. That is, the pointwise inequality $\prod X_i\le\prod|X_i|$ reduces the problem to identifying a comonotonic coupling of the absolute values and coordinating the signs. Theorem~\ref{thm:attainability} (resp.\ \ref{thm:lower-attainability}) shows that the universal upper (resp.\ lower) bound is attained iff the marginal sign-bias vector $\bm p_u$ belongs to the even-parity (resp.\ odd-parity) polytope $P_d^+$ (resp.\ $P_d^-$) for a.e. quantile level $u$. When feasibility holds, a piecewise-affine selection argument converts the pointwise sign weights into an explicit coupling that achieves the bound. For identical marginals, the polytope conditions collapse to scalar thresholds on the common sign-bias function. In this case, the necessary and sufficient conditions are determined by the dimension $d$. Moreover, when $d=3$, we explicitly construct the maximizing and minimizing copulas for arbitrary continuous marginal distributions. Finally, a recursive parity decomposition approach reduces $d$-dimensional polytope membership to a pair of $(d-1)$-dimensional problems. Composing this reduction with the trivariate copulas yields explicit four-dimensional maximizing and minimizing copulas.

The product function $\psi(x_1, \dots, x_d)=\prod_{i=1}^{d}x_i$ is neither supermodular nor submodular in general, so the extremization problems \eqref{eq:md} and \eqref{eq:Md} fall outside the classical Monge--Kantorovich framework in which monotone rearrangements are optimal. For supermodular objectives, the comonotonic coupling is the universal maximizer, but the sign-recombination structure of the product requires a fundamentally different mechanism. The main finding of this paper is that attainability of the universal product bounds can be characterized through parity-polytope feasibility and that the corresponding feasible weights can be used to construct extremal couplings. The problems may nonetheless be cast as multi-marginal optimal transport problems in the sense of \citet{kellerer1984duality} and \citet{pass2015multi}, with the product as cost function. The parity-polytope characterization then provides a structural description of the optimal transport plans.

When the corresponding parity-feasibility conditions fail, the universal bounds are no longer sharp. In that case, the true extremal values $M_d^{\times}$ and $m_d^{\times}$ lie strictly between $-\E\left[\prod_{i=1}^d G_i^{-1}(U)\right]$ and $\E\left[\prod_{i=1}^d G_i^{-1}(U)\right]$. Determining the exact extremal expected products in such cases, together with the dependence structures of the associated optimal couplings, remains an open problem. In addition, our analysis is restricted to absolutely continuous marginal distributions. Extending these results to discrete marginals is another interesting open direction.

\section*{Acknowledgements}

CBW acknowledges financial support from the Natural Sciences and Engineering Research Council of Canada (RGPIN-2025-06879).

\bibliographystyle{apalike}
\bibliography{ref}

\appendix

\section{Examples}\label{sec:examples}

This appendix collects additional examples in dimension $d=3$, illustrating the extremal copula constructions of Section~\ref{sec:d3}. The first two involve shifted exponential marginals with support on $[-a, \infty)$ for some $a>0$; in this case, the maximizer is achievable, but the minimizer is not. The third and fourth treat non-centred normal and Student-t marginals respectively, in which either the maximizer or the minimizer is achievable, depending on the parameters of the marginal distribution.

\subsection{Homogeneous shifted exponentials}\label{subsec:shifted-exp}

Assume $d=3$ and let $X_i+a \sim \mathrm{Exp}(\lambda)$, for $\lambda>0, \ a>0$ and $i=1, 2, 3.$ Then, for $x \ge -a$ and $\quad u\in(0, 1)$, the common marginal cdf and quantile functions are respectively
$$F(x)=1-e^{-\lambda(x+a)}  \text{ and } F^{-1}(u)=-\frac{1}{\lambda}\log(1-u)-a.$$

The following construction gives the maximizing coupling.

\begin{proposition}\label{prop:shifted-exp-maximizer}
Let $X_i\sim F$ be such that $|X_i|\sim G$ for $i=1, 2, 3$, and let $U\eqiid V\sim \U[0, 1]$. Let $u_0=1-e^{-\lambda a}$ and $u_a=1-c$ with $c=e^{-2\lambda a}$. Iff $a\le \frac{\log 2}{2\lambda}$, the universal upper bound
$M_3^{\times}=\E\left[G^{-1}(U)^3\right]$ is attained by $X_i:=F^{-1}(U_i)$ with 
\begin{equation}\label{eq:shiftedexp-U123}
U_1=U, \qquad
U_2=J U+(1-J) \tau(U), \qquad
U_3=I U_2+(1-I)\left(J \tau(U)+(1-J) U\right),
\end{equation}where $I=\id_{\{U>u_0\}}$, $J=\id_{\{V\le s(U)\}}$ with the mixing function
\begin{equation}\label{eq:shiftedexp-s}
s(u)=
\begin{cases}
\dfrac{1}{2}, & 0\le u\le u_0, \\[0.6em]
1-\dfrac{c}{2(1-u)^2}, & u_0<u\le u_a, \\[0.9em]
1, & u_a<u\le 1,
\end{cases}
\end{equation} and 
\begin{equation}\label{eq:shiftedexp-f}
\tau(u):=
  \begin{cases}
  1-\dfrac{c}{1-u}, & u\in[0, u_a], \\[0.8em]
  u, & u\in(u_a, 1].
  \end{cases}
\end{equation}
\end{proposition}

\begin{proof}
The construction \eqref{eq:shiftedexp-U123} is the specialization of the general trivariate maximizer \eqref{eq:d3-maximizer-copula} to identical shifted exponential marginals, with the sign-flip map $\tau$ from \eqref{eq:shiftedexp-f} and the mixing function $s$ from \eqref{eq:shiftedexp-s}. By Proposition~\ref{cor:iid}, the upper bound is attainable iff $p(u)\ge 1/3$ for a.e.\ $u$. For $u>u_a$, we have $p(u)=1$ since $F^{-1}(u)>a$ has no negative counterpart of the same magnitude. For $u\in(0, u_a)$, the sign bias $p(u)=f(y)/(f(y)+f(-y))$ is minimized at $u=u_a$, where $p(u_a)=1/(1+e^{2\lambda a})$. Hence $p(u)\ge 1/3$ for all relevant levels iff $e^{2\lambda a}\le 2$, equivalently $a\le \frac{\log 2}{2\lambda}$. On the positive branch $u\in(u_0, u_a]$, the branch density ratio
$$r(u):=\frac{f\left(-F^{-1}(u)\right)}{f\left(F^{-1}(u)\right)}=\frac{c}{(1-u)^2}$$
satisfies $s(u)=1-r(u)/2$. Since $r(u)$ is increasing on $(u_0, u_a]$ and $r(u_a)=e^{2\lambda a}$, the same condition gives $s(u)\ge 0$ on the active branch. Thus, the construction produces the correct marginals and attains the bound exactly under the stated condition.
\end{proof}
For every $u\in(0, u_a)$, $F^{-1}(\tau(u))=-F^{-1}(u),$ so choosing $U$ versus $\tau(U)$ flips the sign of the quantile while preserving its absolute value. Moreover, $\tau$ is an involution on $(0, u_a)$, that is, $\tau(\tau(u))=u,$ for all $u\in(0, u_a),$ and $u_0$ is a fixed point such that $\tau(u_0)=u_0$.

The copula for the example in this section has three regions, as we can see in Figure~\ref{fig:shiftedexp-maximizer}. The break at $u_0=F(0)$ marks where $X_1$ changes sign: for $u\le u_0$, $X_1\le 0$ and the support splits equally ($s=1/2$) between the $(-, -, +)$ and $(-, +, -)$ legs, to guarantee that one of $X_2$ and $X_3$ is positive and the other negative, so that the product $X_1X_2X_3$ is positive. For $u\in(u_0, u_a]$, $X_1>0$ and the support alternates between the $(+, +, +)$ and $(+, -, -)$ legs with probability $s(u)$ and $1-s(u)$. The break at $u_a=F(a)$ marks where the negative part of the support ends. That is, for $u>u_a$, the magnitude $|X_1|=F^{-1}(u)>a$ exceeds the lower bound of the support, so no negative realization of the same magnitude exists. In that case, we have $p(u)=1$, and the support lies entirely on the $(+, +, +)$ leg. This truncation shortens the legs relative to a distribution supported on all of $\R$; see Appendix \ref{subsec:normal-mu075} for an example with non-truncated support.

\begin{figure}
\centering
\begin{subfigure}[b]{0.47\textwidth}
  \centering
  \includegraphics[width=\textwidth]{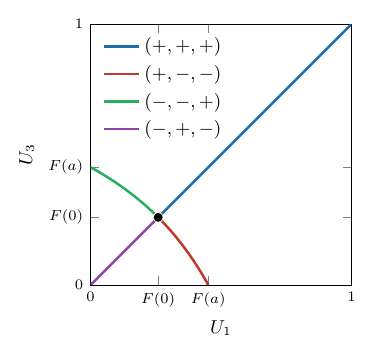}
  \caption{Marginal copula $(U_1, U_3)$.}
\end{subfigure}
\hfill
\begin{subfigure}[b]{0.47\textwidth}
  \centering
  \includegraphics[width=\textwidth]{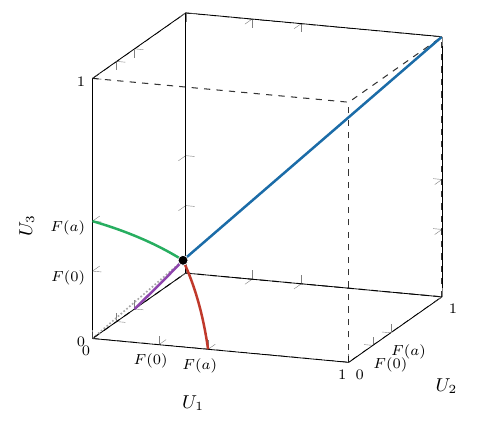}
  \caption{Support in $[0, 1]^3$.}
\end{subfigure}
\caption{Support of the maximizing copula for shifted exponential marginals ($a=0.30$, $\lambda=1$). Each colour corresponds to one even-parity sign pattern; the black dot marks the junction $(u_0, u_0, u_0)$. Left: the marginal copula $(U_1, U_3)$, showing the three regions of support; the two other marginal copulas $(U_1, U_2)$ and $(U_2, U_3)$ have the same structure by symmetry. Right: the full support in $[0, 1]^3$.}
\label{fig:shiftedexp-maximizer}
\end{figure}

\subsection{Heterogeneous shifted exponentials}\label{subsec:shifted-exp2}

We now consider an example where each marginal distribution is different, so that the sign-bias functions $p_i(u)$ are distinct, and the sign weights $w_{\bm s}(u)$ are nontrivial functions of $u$. Let $d=3$. For $i=1, 2, 3$, assume
$X_i+a_i \sim \mathrm{Exp}(\lambda_i),$ where $\lambda_i>0,$ and $ a_i>0,$
then
$$F_i(x)=1-e^{-\lambda_i(x+a_i)}, \text{ and } f_i(x)=\lambda_i e^{-\lambda_i(x+a_i)}, \text{ where } \quad x\ge -a_i.$$

A direct calculation gives, for $y\ge 0$,
\begin{equation}\label{eq:Gi-shifted-exp}
G_i(y)=\p(|X_i|\le y)=
\begin{cases}
F_i(y)-F_i(-y)=2e^{-\lambda_i a_i}\sinh(\lambda_i y), & 0\le y\le a_i, \\[0.4em]
F_i(y)=1-e^{-\lambda_i(y+a_i)}, & y>a_i,
\end{cases}
\end{equation}
and hence, with
$u_{a, i}:=G_i(a_i)=1-e^{-2\lambda_i a_i}.$
The quantile function admits the closed-form expression
\begin{equation}\label{eq:yi-quantile-shifted-exp}
G_i^{-1}(u)=
\begin{cases}
\displaystyle \frac{1}{\lambda_i} \sinh^{-1} \left(\frac{u}{2}e^{\lambda_i a_i}\right),
& 0<u<u_{a, i}, \\[1.0em]
\displaystyle -\frac{1}{\lambda_i}\log(1-u)-a_i,
& u_{a, i}\le u<1.
\end{cases}
\end{equation}

Since $f_i(-y)=\lambda_i e^{-\lambda_i(a_i-y)}$ for $0<y\le a_i$ and $f_i(-y)=0$ for $y>a_i$, the sign bias \eqref{eq:sign-bias} is
\begin{equation}\label{eq:pi-u-shifted-exp}
p_i(u)=\frac{f_i(G_i^{-1}(u))}{f_i(G_i^{-1}(u))+f_i(-G_i^{-1}(u))}=
\begin{cases}
\displaystyle \frac{1}{1+e^{2\lambda_i G_i^{-1}(u)}}, & 0<u<u_{a, i}, \\[0.8em]
1, & u_{a, i}\le u<1.
\end{cases}
\end{equation}Then, the even-parity sign weights in \eqref{eq:d3-even-weights} can be calculated with $p_i(u)$ in \eqref{eq:pi-u-shifted-exp}. Let $U\eqiid V\sim \U[0, 1]$. Set $Y_i:=G_i^{-1}(U)$ and sample an even-parity sign vector $\bm S=(S_1, S_2, S_3)\in E_3$ with $\p(\bm S=\bm s\mid U)=w_{\bm s}(U)$ for each $\bm s\in E_3$, using $V$ as the randomization source. For $i=1, 2, 3$, define
\begin{equation}\label{eq:Xi-hetero-exp}
X_i:=S_i Y_i.
\end{equation}

\begin{proposition}\label{prop:shifted-exp-hetero}
If the four weights \eqref{eq:d3-even-weights} are nonnegative for a.e.\ $u\in(0, 1)$, then the coupling \eqref{eq:Xi-hetero-exp} attains the universal upper bound
$M_3^{\times}=\E\left[\prod_{i=1}^3 G_i^{-1}(U)\right].$
\end{proposition}

\begin{proof}
The construction is the specialization of Theorem~\ref{thm:attainability} to $d=3$ with the sign weights \eqref{eq:d3-even-weights}. 
\end{proof}

For illustration, let $(\lambda_1, \lambda_2, \lambda_3)=(0.8, 1.0, 1.9)$ and $(a_1, a_2, a_3)=(0.15, 0.38, 0.20).$
Then $(\lambda_1 a_1, \lambda_2 a_2, \lambda_3 a_3)=(0.12, 0.38, 0.38)$ and $(u_{a, 1}, u_{a, 2}, u_{a, 3})\approx(0.213, 0.532, 0.532).$
A numerical evaluation of \eqref{eq:d3-even-weights} confirms that all four weights are nonnegative for every $u\in(0, 1)$, so the assumptions of Proposition~\ref{prop:shifted-exp-hetero} are satisfied. 

We present the support of the maximizing copula in Figure~\ref{fig:shiftedexp-hetero-maximizer}. The support has three regions, as in the homogeneous case, but the boundaries are no longer aligned across the three coordinates. The first break at $\min(F_1(0), F_2(0), F_3(0)) = u_{0, 1}$ corresponds to the point where $X_1$ changes sign, and the second break at $\max(F_1(a_1), F_2(a_2), F_3(a_3)) = u_{a, 3}$ marks the last magnitude level at which a negative realisation is available, here for the $X_3$ coordinate; beyond this level all three coordinates are forced to be positive. The break for $X_2$ lies somewhere in between, where the blue and red curves change behaviour. Since each support is different, the marginal copulas $(U_1, U_2)$, $(U_1, U_3)$ and $(U_2, U_3)$ have different structures, but they all share the same three-region structure.

The marginal copula $(U_1, U_3)$ in Figure~\ref{fig:shiftedexp-hetero-maximizer}(a) illustrates the interplay between the comonotonicity of magnitudes and the asymmetry of the marginals. The four legs of the support meet at the junction $(F_1(0), F_3(0))$ and separate into two pairs, one on each side. To the left of the junction, $X_1<0$ on both the green $(-, -, +)$ and purple $(-, +, -)$ legs, so both share $U_1=F_1(-y_1(u))$, which decreases from $F_1(0)$ to $0$ as the magnitude level $u$ increases from $0$ to $G_1(a_1)$. At each level the magnitude $|X_3|=y_3(u)$ is the same on both legs, but $X_3$ takes opposite signs, giving $U_3=F_3(y_3(u))$ on the green leg and $U_3=F_3(-y_3(u))$ on the purple leg. The two curves therefore diverge from $F_3(0)$ in opposite directions along the $U_3$-axis, reaching $F_3(G_3^{-1}(G_1(a_1)))$ and $F_3(-G_3^{-1}(G_1(a_1)))$ at $U_1=0$. Because $F_3$ is not symmetric around zero, i.e., $F_3(y)-F_3(0)\neq F_3(0)-F_3(-y)$ for $y>0$, the two branches separate by unequal amounts. To the right of the junction, the blue $(+, +, +)$ and red $(+, -, -)$ legs form an analogous pair with $X_1>0$ and $X_3$ of opposite signs, diverging from $F_3(0)$ along the $U_3$-axis as $U_1$ increases from $F_1(0)$ toward $1$. For $U_1>F_3(a_3)$, the magnitude $y_3(u)$ exceeds $a_3$ and $F_3(-y_3(u))=0$, so no negative realisation of $X_3$ with the same absolute value exists; the red leg terminates and the support reduces to the single $(+, +, +)$ leg, which lies on the diagonal $U_1=U_2=U_3$ and corresponds to the comonotonic copula.

\begin{figure}
\centering
\begin{subfigure}[b]{0.47\textwidth}
  \centering
  \includegraphics[width=\textwidth]{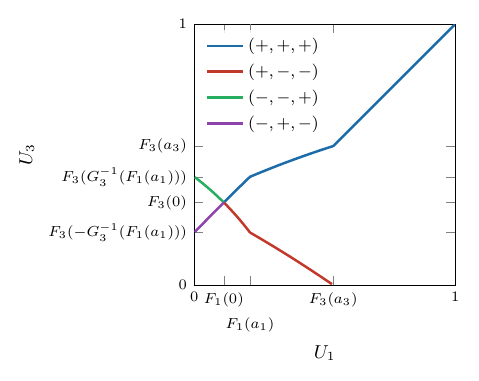}
  \caption{Marginal copula $(U_1, U_3)$.}
\end{subfigure}
\hfill
\begin{subfigure}[b]{0.47\textwidth}
  \centering
  \includegraphics[width=\textwidth]{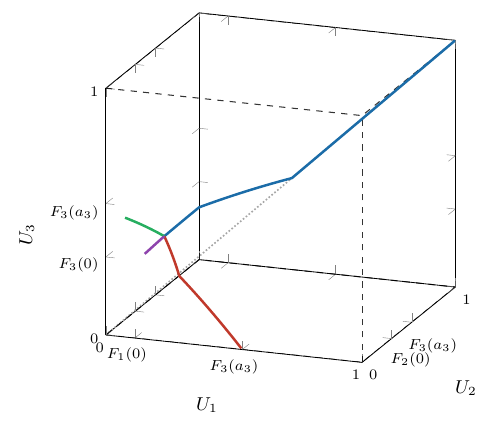}
  \caption{Support in $[0, 1]^3$.}
\end{subfigure}
\caption{Support of the maximizing copula for heterogeneous shifted exponential marginals with $(\lambda_1, \lambda_2, \lambda_3)=(0.8, 1.0, 1.9)$ and $(a_1, a_2, a_3)=(0.15, 0.38, 0.20)$. The four coloured curves correspond to the even-parity sign patterns $(+, +, +)$, $(+, -, -)$, $(-, -, +)$ and $(-, +, -)$.}
\label{fig:shiftedexp-hetero-maximizer}
\end{figure}

\subsection{Non-centred normal marginals}\label{subsec:normal-mu075}

Let $\Phi$ and $\phi$ denote the standard normal cdf and density. As a third example, we consider identical normal marginals $F_i=F$ with mean $\mu\in\R$ and variance $\sigma^2>0$, so that $F(x)=\Phi\left((x-\mu)/\sigma\right)$ and $F^{-1}(u)=\mu+\sigma\Phi^{-1}(u)$. Unlike the shifted exponential, the normal distribution is supported on all of~$\R$, so the sign bias $p(u)\in(0, 1)$ for every $u\in(0, 1)$. In particular, every magnitude level $u$ admits both a positive and a negative realization of the same absolute value, so the coupling must choose between two or more even-parity legs at every level. A nonzero mean introduces asymmetry in the sign bias: positive values of $\mu$ make positive signs more likely in the tail, and negative values have the opposite effect. We show that when $\mu\ge 0$ the universal upper bound is attained, and when $\mu\le 0$ the universal lower bound is attained; the copulas of the two extremizers are related by a reversal of the mixing function.

Define the sign threshold $u_0:=F(0)=\Phi\left(-\mu/\sigma\right)$, so that $F^{-1}(u)\le 0$ iff $u\le u_0$. When $\mu>0$ we have $u_0<1/2$; when $\mu<0$ we have $u_0>1/2$; and $u_0=1/2$ when $\mu=0$. Define the sign-flip involution $\tau:(0, 1)\to(0, 1)$ by
$\tau(u):=F\left(-F^{-1}(u)\right)=\Phi\left(-\Phi^{-1}(u)-\frac{2\mu}{\sigma}\right).$
Then $\tau$ is a strictly decreasing involution with fixed point $u_0$, that is, $\tau(\tau(u))=u$ and $\tau(u_0)=u_0$, and it flips the quantile around zero: $F^{-1}(\tau(u))=-F^{-1}(u)$ for all $u\in(0, 1)$. The branch density ratio is
$$r_\mu(u):=\frac{f\left(-F^{-1}(u)\right)}{f\left(F^{-1}(u)\right)}=\exp\!\left(-\frac{2\mu}{\sigma}\,\Phi^{-1}(u)-\frac{2\mu^2}{\sigma^2}\right).$$
This ratio depends on $\mu$ and $\sigma$ only through $\mu/\sigma$, so the sign structure of the extremal coupling is invariant to scale and depends only on the signal-to-noise ratio $\mu/\sigma$. A direct computation gives $r_\mu(u)\le 1$ on $(u_0, 1)$ when $\mu\ge 0$, and on $(0, u_0]$ when $\mu\le 0$.

The sign bias simplifies to the logistic form
\begin{equation}\label{eq:normal-sign-bias}
p(u)=\frac{\exp\!\left(2\mu\, G^{-1}(u)/\sigma^2\right)}{1+\exp\!\left(2\mu\, G^{-1}(u)/\sigma^2\right)}.
\end{equation}
where $G^{-1}(u)$ is the absolute-value quantile.

\begin{proposition}\label{prop:normal-maximizer}
Let $X \sim\mathrm{N}(\mu, \sigma^2)$ with $\mu\in\R$ and $\sigma>0$, and let $\tau$ and $u_0$ be as above. Let $U\eqiid V\sim\U[0, 1]$, and set $I:=\id_{\{U>u_0\}}$, and $G$ be the cdf of $|X|$.
\begin{enumerate}[(1)]
\item If $\mu\ge 0$, the universal upper bound $M_3^{\times}=\E\left[G^{-1}(U)^3\right]$ is attained by $X_i=F^{-1}(U_i)$ with $$U_1=U, \quad U_2=J^+ U+(1-J^+) \tau(U), \quad U_3=I U_2+(1-I)\left(J^+ \tau(U)+(1-J^+) U\right),$$ where $J^+=\id_{\{V\le s^+(U)\}}$ with
\begin{equation}\label{eq:normal-s}
s^+(u)=
\begin{cases}
\displaystyle \frac12, & 0<u\le u_0, \\[0.8em]
\displaystyle 1-\frac12\, r_\mu(u), & u_0<u<1,
\end{cases}
\end{equation}

\item If $\mu\le 0$, the universal lower bound $m_3^{\times}=-\E\left[G^{-1}(U)^3\right]$ is attained by $X_i=F^{-1}(U_i)$ with \begin{equation}\label{eq:normal-U123-min}
U_1=U, \quad
U_2=J^- U+(1-J^-) \tau(U), \quad
U_3=(1-I) U_2+I\left(J^- \tau(U)+(1-J^-) U\right),
\end{equation} where $J^-=\id_{\{V\le s^-(U)\}}$ with
\begin{equation}\label{eq:normal-s-min}
s^-(u)=
\begin{cases}
\displaystyle 1-\frac12\, r_\mu(u), & 0<u\le u_0, \\[0.8em]
\displaystyle \frac12, & u_0<u<1.
\end{cases}
\end{equation}
\end{enumerate}
\end{proposition}

\begin{proof}
By Proposition~\ref{cor:iid} with $d=3$, the upper bound is attainable iff $p(u)\ge 1/3$ for a.e.\ $u$. When $\mu\ge 0$, the logistic form \eqref{eq:normal-sign-bias} gives $p(u)\ge 1/2\ge 1/3$ for all $u\in(0, 1)$, since $\mu G^{-1}(u)/\sigma^2\ge 0$. Theorem~\ref{thm:attainability} then guarantees that the specialization of~\eqref{eq:d3-maximizer-copula} to identical marginals produces the correct marginals and attains the bound; the mixing function $s^+$ in~\eqref{eq:normal-s} is obtained by matching the even-parity weights~\eqref{eq:d3-even-weights} as in Section~\ref{subsec:d3-copula}. On the positive branch, $r_\mu(u)=(1-p(w))/p(w)\le 1$ since $p(u)\ge 1/2$, so $s^+(u)=1-r_\mu(u)/2\in[1/2, 1)$. Similarly, for the lower bound, Proposition~\ref{cor:iid} requires $p(u)\le 2/3$ for a.e.\ $u$. When $\mu\le 0$, \eqref{eq:normal-sign-bias} gives $p(u)\le 1/2\le 2/3$, so Theorem~\ref{thm:lower-attainability} applies. The minimiser copula~\eqref{eq:normal-U123-min} is obtained by swapping $I$ and $1-I$ in $U_3$ as in Section~\ref{subsec:d3-copula}. On the negative branch, $r_\mu(u)=p(w)/(1-p(w))\le 1$ since $p(u)\le 1/2$, so $s^-(u)=1-r_\mu(u)/2\in[1/2, 1)$.
\end{proof}

The two constructions exhibit a mirror symmetry: the maximiser mixing function $s^+$ places the nontrivial branch on $(u_0, 1)$, where $F^{-1}(u)>0$ and the coupling must decide how to split the positive realisation between $(+, +, +)$ and $(+, -, -)$; the minimiser mixing function $s^-$ places the nontrivial branch on $(0, u_0]$, where $F^{-1}(u)\le 0$ and the coupling must split the negative realisation between $(-, +, +)$ and $(-, -, -)$. This symmetry reflects the identity $\inf_{X_i\sim \mathrm{N}(\mu, \sigma^2)}\E[X_1X_2X_3]=-\sup_{Y_i\sim \mathrm{N}(-\mu, \sigma^2)}\E[Y_1Y_2Y_3]$, which holds because replacing each $X_i$ by $-X_i$ negates the product and reverses the sign of the mean.

When $\mu=0$, the normal distribution is symmetric, $u_0=1/2$, $\tau(u)=1-u$, and $r_\mu(u)=1$ for all $u$, so both mixing functions reduce to $s^+(u)=s^-(u)=1/2$ everywhere. Both the upper and lower bounds are attainable, and the extremal couplings both place $U_2, U_3 \in \{U, 1-U\}$ with a fair coin, recovering the symmetric case solved by \citet{bernard2023coskewness}.

Figure~\ref{fig:normal-maximizer} shows the support of the maximising copula for $\mu=\Phi^{-1}(0.75)\approx 0.674$ and $\sigma=1$ (so $u_0=0.25$), and Figure~\ref{fig:normal-minimizer} shows the minimising copula for $\mu=\Phi^{-1}(0.25)\approx -0.674$ and $\sigma=1$ (so $u_0=0.75$). The two figures are reflections of each other along the main diagonal of $[0, 1]^3$, with the junction point moving from $(0.25, 0.25, 0.25)$ to $(0.75, 0.75, 0.75)$ and the sign patterns switching from even to odd parity.

\begin{figure}
\centering
\begin{subfigure}[b]{0.47\textwidth}
  \centering
  \includegraphics[width=\textwidth]{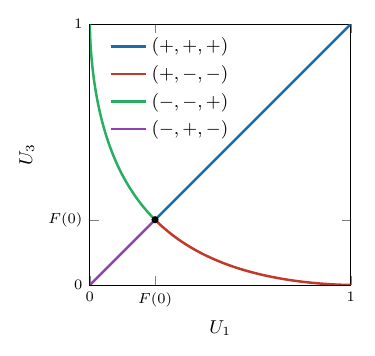}
  \caption{Marginal copula $(U_1, U_3)$.}
\end{subfigure}
\hfill
\begin{subfigure}[b]{0.47\textwidth}
  \centering
  \includegraphics[width=\textwidth]{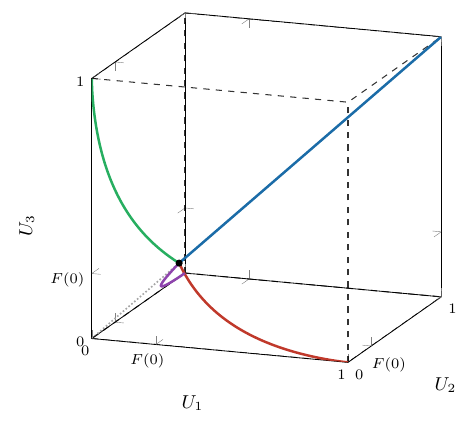}
  \caption{Support in $[0, 1]^3$.}
\end{subfigure}
\caption{Support of the maximizing copula for $\mathrm{N}(\mu, 1)$ marginals ($\mu=\Phi^{-1}(0.75)\approx 0.674$). Each colour corresponds to one even-parity sign pattern; the black dot marks the junction $(u_0, u_0, u_0)=(0.25, 0.25, 0.25)$.}
\label{fig:normal-maximizer}
\end{figure}

\begin{figure}
\centering
\begin{subfigure}[b]{0.47\textwidth}
  \centering
  \includegraphics[width=\textwidth]{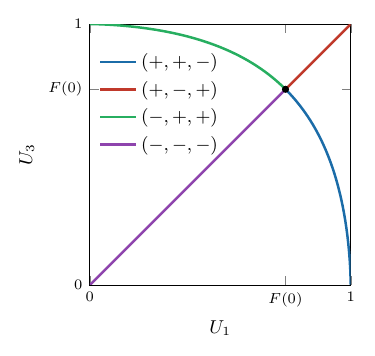}
  \caption{Marginal copula $(U_1, U_3)$.}
\end{subfigure}
\hfill
\begin{subfigure}[b]{0.47\textwidth}
  \centering
  \includegraphics[width=\textwidth]{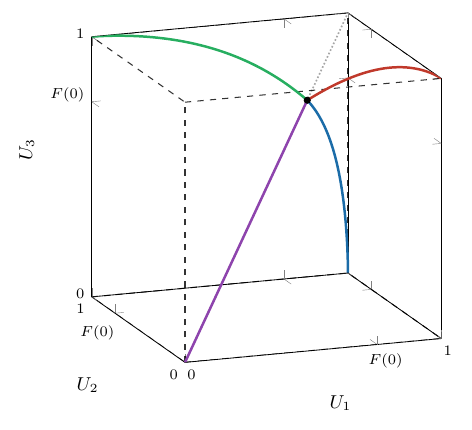}
  \caption{Support in $[0, 1]^3$.}
\end{subfigure}
\caption{Support of the minimizing copula for $\mathrm{N}(\mu, 1)$ marginals ($\mu=\Phi^{-1}(0.25)\approx -0.674$). Each colour corresponds to one odd-parity sign pattern; the black dot marks the junction $(u_0, u_0, u_0)=(0.75, 0.75, 0.75)$.}
\label{fig:normal-minimizer}
\end{figure}

For the shifted exponential marginals of Appendices~\ref{subsec:shifted-exp} and~\ref{subsec:shifted-exp2}, Corollary~\ref{cor:skew-obstruction} rules out the universal lower bound directly. For normal marginals with $\mu\ne 0$, unbounded support does not suffice to attain both extremal bounds: when $\mu>0$, $p(u)\to 1$ as $u\to 1$, which violates the odd-parity threshold $p(u)\le 2/3$; when $\mu<0$, $p(u)\to 0$, which violates the even-parity threshold $p(u)\ge 1/3$. Both bounds are attainable simultaneously iff $p(u)\in[1/3, 2/3]$ for almost every~$u$, and for the normal distribution, this holds only when $\mu=0$.

\subsection{Non-centred Student-t marginals}\label{subsec:student-t}

Let $T_\nu$ and $t_\nu$ denote the cdf and density of the standard Student-$t$ distribution with $\nu>3$ degrees of freedom. In this example, we consider identically distributed Student-$t$ marginals $F_i=F$ with location parameter $\mu\in\R$ and scale parameter $\sigma>0$, so that $X=\mu+\sigma Z$, where $Z\sim t_\nu$. Recall that for the normal distribution, one tail eventually dominates the other, preventing simultaneous membership in both polytopes. However, for the Student-\$t\$ distribution, the positive and negative tails have the same asymptotic decay rate, so membership in both polytopes is possible. We show that both universal bounds may therefore be attained when $\mu\ne0$.

Write $a:=\mu/\sigma$. The common marginal cdf and sign threshold are
$$F(x)=T_\nu\left(\frac{x-\mu}{\sigma}\right), \qquad u_0:=F(0)=T_\nu(-a).$$
The sign-flip involution is
$$\tau(u):=F\left(-F^{-1}(u)\right) =T_\nu\left(-T_\nu^{-1}(u)-2a\right),$$
and, on writing $z=T_\nu^{-1}(u)$, the corresponding branch density ratio is
\begin{equation}\label{eq:t-density-ratio}
r_{a, \nu}(u):=\frac{f(-F^{-1}(u))}{f(F^{-1}(u))}
=\left(\frac{\nu+z^2}{\nu+(z+2a)^2}\right)^{(\nu+1)/2}.
\end{equation}

Let $G$ be the cdf of $|X|$, and write $z:=G^{-1}(u)/\sigma$. The sign bias is
\begin{equation}\label{eq:t-sign-bias}
p(u)=\frac{f(G^{-1}(u))}{f(G^{-1}(u))+f(-G^{-1}(u))}
=\left[1+
\left(\frac{\nu+(z-a)^2}{\nu+(z+a)^2}\right)^{(\nu+1)/2}
\right]^{-1}.
\end{equation}
Thus $p(u)\ge1/2$ for every $u\in(0, 1)$ when $\mu\ge0$, while $p(u)\le1/2$ for every $u\in(0, 1)$ when $\mu\le0$. Define
\begin{equation}\label{eq:t-shift-threshold}
a_\nu^\star:=\sqrt{\nu}\,
\sinh\left(\frac{\log 2}{\nu+1}\right).
\end{equation}
The following construction gives the maximizing and minimizing couplings.

\begin{proposition}\label{prop:t-extremizers}
Let $X\sim F$, and let $\tau$, $u_0$, and $r_{a, \nu}$ be as above. Let $U\eqiid V\sim\U[0, 1]$, set $I:=\id_{\{U>u_0\}}$, and let $G$ be the cdf of $|X|$.
\begin{enumerate}[(1)]
\item The universal upper bound is attainable iff $a\ge-a_\nu^\star$. Under this condition,
$$M_3^\times=\E[G^{-1}(U)^3]$$
is attained by $X_i=F^{-1}(U_i)$, where
$$U_1=U, \qquad U_2=J^+U+(1-J^+)\tau(U), \qquad U_3=IU_2+(1-I)\bigl(J^+\tau(U)+(1-J^+)U\bigr),$$
$J^+=\id_{\{V\le s^+(U)\}}$, and
$$
s^+(u)=
\begin{cases}
\dfrac12, & 0<u\le u_0, \\[0.6em]
1-\dfrac12r_{a, \nu}(u), & u_0<u<1.
\end{cases}
$$

\item The universal lower bound is attainable iff $a\le a_\nu^\star$. Under this condition,
$$m_3^\times=-\E[G^{-1}(U)^3]$$
is attained by $X_i=F^{-1}(U_i)$, where
$$U_1=U, \qquad U_2=J^-U+(1-J^-)\tau(U), \qquad U_3=(1-I)U_2+I\bigl(J^-\tau(U)+(1-J^-)U\bigr),$$
$J^-=\id_{\{V\le s^-(U)\}}$, and
$$
s^-(u)=
\begin{cases}
1-\dfrac12r_{a, \nu}(u), & 0<u\le u_0, \\[0.6em]
\dfrac12, & u_0<u<1.
\end{cases}
$$
\end{enumerate}
Consequently, both universal bounds are attainable iff $|\mu|/\sigma\le a_\nu^\star$.
\end{proposition}

\begin{proof}
Suppose first that $a\ge0$. By \eqref{eq:t-sign-bias}, $p(u)\ge1/2$, so the upper-bound condition $p(u)\ge1/3$ holds automatically. For the lower bound, the sign odds are
$$\frac{p(u)}{1-p(u)} =\left(\frac{\nu+(z+a)^2}{\nu+(z-a)^2}\right)^{(\nu+1)/2}.$$
The ratio inside the power has the derivative
$$\frac{4a(\nu+a^2-z^2)}{[\nu+(z-a)^2]^2},$$
so it is maximized over $z\ge0$ at $z^\star=\sqrt{\nu+a^2}$. At this point,
$$\frac{\nu+(z^\star+a)^2}{\nu+(z^\star-a)^2} =\frac{\sqrt{\nu+a^2}+a}{\sqrt{\nu+a^2}-a} =\exp\left(2\operatorname{arsinh}\frac{a}{\sqrt\nu}\right).$$
It follows that
$$\sup_{u\in(0, 1)}\frac{p(u)}{1-p(u)} =\exp\left((\nu+1)\operatorname{arsinh}\frac{a}{\sqrt\nu}\right).$$
By Proposition~\ref{cor:iid}, the lower bound is attainable iff $p(u)\le2/3$ for every $u\in(0, 1)$, equivalently iff this supremum is at most $2$. This gives $a\le a_\nu^\star$. For $a\le0$, reflection gives $p_a(u)=1-p_{-a}(u)$. Hence, the lower bound is attained automatically, while the upper bound is attainable iff $a\ge-a_\nu^\star$. The copula formulas follow directly from \eqref{eq:d3-maximizer-copula} and \eqref{eq:d3-minimizer-copula}. Under the respective threshold conditions, their mixing functions take values in $[0, 1]$.
\end{proof}

Unlike the non-centred normal case, both the minimizer and maximizer can be achieved from the same set of marginal distributions. Indeed, the polynomial tails of the Student-$t$ distribution imply that $p(u)\to1/2$ as $u\to1$. If the standardized shift \$|\textbackslash{}mu|/\textbackslash{}sigma\$ is sufficiently small, the sign bias remains in \$[1/3,2/3]\$ at every magnitude level.

For illustration, fix $\nu=5$, $\sigma=1$, and $\mu=1/4$. Then $a_5^\star\approx0.259>1/4,$ so both universal bounds are attainable for the same asymmetric marginal distribution. The maximum sign bias occurs at $z^\star=9/4,$ where
$$\sup_{u\in(0, 1)}\frac{p(u)}{1-p(u)} =\left(\frac54\right)^3=\frac{125}{64}, \qquad \sup_{u\in(0, 1)}p(u)=\frac{125}{189}<\frac23.$$
Hence $1/2\le p(u)\le125/189$ for every $u\in(0, 1)$. Moreover, $u_0=T_5(-1/4)\approx0.406$ and $\E[|X|^3]\approx4.92$. It follows that $M_3^\times\approx4.92$ and $m_3^\times\approx-4.92$.

Figure~\ref{fig:student-t-maximizer} shows the support of the maximizing copula, and Figure~\ref{fig:student-t-minimizer} shows the support of the minimizing copula. Both figures use the same non-centred Student-$t$ marginal distribution. The maximizing support has four even-parity legs, while the minimizing support has four odd-parity legs. All four legs meet at $(u_0, u_0, u_0)$ in each figure.

\begin{figure}[htbp]
\centering
\begin{subfigure}[b]{0.47\textwidth}
  \centering
  \includegraphics[width=\textwidth]{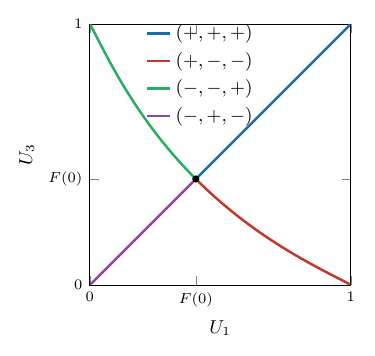}
  \caption{Marginal copula $(U_1, U_3)$.}
\end{subfigure}
\hfill
\begin{subfigure}[b]{0.47\textwidth}
  \centering
  \includegraphics[width=\textwidth]{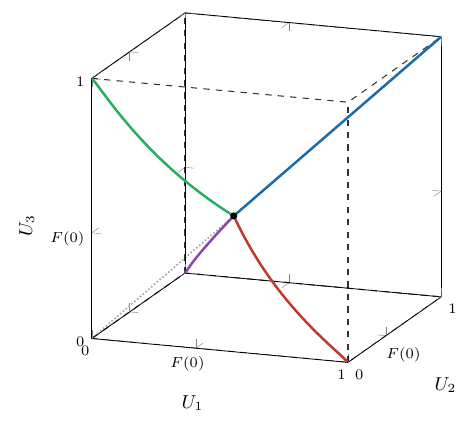}
  \caption{Support in $[0, 1]^3$.}
\end{subfigure}
\caption{Support of the maximizing copula for shifted Student-$t$ marginals with $\nu=5$, $\mu=1/4$, and $\sigma=1$. Each colour corresponds to one even-parity sign pattern; the black dot marks the junction $(u_0, u_0, u_0)$, where $u_0=T_5(-1/4)\approx0.406$.}
\label{fig:student-t-maximizer}
\end{figure}

\begin{figure}[htbp]
\centering
\begin{subfigure}[b]{0.47\textwidth}
  \centering
  \includegraphics[width=\textwidth]{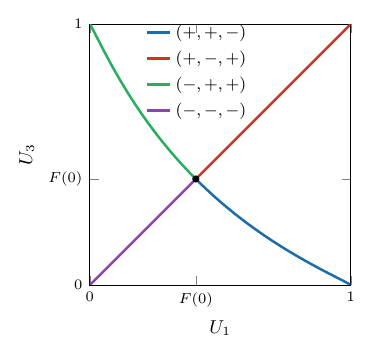}
  \caption{Marginal copula $(U_1, U_3)$.}
\end{subfigure}
\hfill
\begin{subfigure}[b]{0.47\textwidth}
  \centering
  \includegraphics[width=\textwidth]{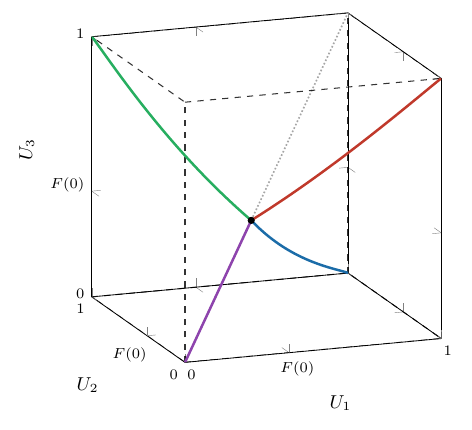}
  \caption{Support in $[0, 1]^3$.}
\end{subfigure}
\caption{Support of the minimizing copula for the same shifted Student-$t$ marginals as in Figure~\ref{fig:student-t-maximizer}. Each colour corresponds to one odd-parity sign pattern; the black dot marks the junction $(u_0, u_0, u_0)$.}
\label{fig:student-t-minimizer}
\end{figure}

\section{Omitted proofs}\label{app:proofs}
\begin{proof}[Proof of Lemma~\ref{lem:product-bounds}]
For any coupling $(X_1, \dots, X_d)$ with the prescribed marginals, the pointwise inequalities above and Lemma~\ref{lem:nonneg-product} give
$$\left|\E\left[\prod_{i=1}^d X_i\right]\right|\le \E\left[\prod_{i=1}^d |X_i|\right]\le \E\left[\prod_{i=1}^d G_i^{-1}(U)\right].$$
Taking the supremum and infimum over all couplings yields \eqref{eq:upper-bound} and \eqref{eq:lower-bound}. If the absolute values are comonotonic, the second inequality is an equality. If, in addition, the product is nonnegative (resp.\ nonpositive) a.s., then $\prod_{i=1}^d X_i=\prod_{i=1}^d |X_i|$ (resp.\ $\prod_{i=1}^d X_i=-\prod_{i=1}^d |X_i|$) a.s., which gives equality in the upper (resp.\ lower) bound.

Conversely, suppose that the right-hand side of \eqref{eq:upper-bound} is finite and equality holds in \eqref{eq:upper-bound} (resp.\ \eqref{eq:lower-bound}). Let $(X_1, \dots, X_d)$ be any coupling attaining $M_d^{\times}$ (resp.\ $m_d^{\times}$). Both inequalities in the displayed chain are then equalities, so the absolute-product expectation equals the comonotonic value. The nonnegative integrable random variable $\prod_{i=1}^d |X_i|-\prod_{i=1}^d X_i$ (resp.\ $\prod_{i=1}^d |X_i|+\prod_{i=1}^d X_i$) has expectation zero and is therefore zero a.s., which gives the required sign condition. If all $G_i$ are continuous, equality in the second inequality also implies comonotonicity of $(|X_1|, \dots, |X_d|)$ by Lemma~\ref{lem:nonneg-product}.
\end{proof}

\begin{proof}[Proof of Lemma \ref{lem:measurable-selection}]
We prove part (1) for the upper bound, and the lower bound will follow from a similar argument. The goal is to assign, to every feasible vector $\bm x\in P_d^+$, a probability for each sign pattern in $E_d$. These probabilities must vary measurably with $\bm x$ so that we may later substitute $\bm x=\bm p_u$. The argument is a finite-dimensional version of the measurable-selection construction behind the Kuratowski--Ryll-Nardzewski theorem; see, for example, \citet[Theorem~18.13]{aliprantis2006infinite}.

Write $m:=2^{d-1}$ and fix an ordering $E_d=\{\bm s^{(1)}, \dots, \bm s^{(m)}\}$. Let
$$\Delta_m:=\left\{\bm w\in[0, 1]^m:\sum_{j=1}^m w_j=1\right\}.$$
A vector $\bm w\in\Delta_m$ can be viewed as a probability distribution on $E_d$: the sign pattern $\bm s^{(j)}$ receives probability $w_j$. By the definition of $P_d^+$, a vector $\bm x$ belongs to $P_d^+$ precisely when there is some $\bm w\in\Delta_m$ such that
$$\bm x=\sum_{j=1}^m w_jv(\bm s^{(j)}).$$
Since the $i$th coordinate of $v(\bm s^{(j)})$ is $1$ exactly when $s_i^{(j)}=+1$, the $i$th coordinate of this identity is
$$x_i=\sum_{\substack{1\le j\le m\\s_i^{(j)}=+1}}w_j.$$
Thus, $x_i$ is the probability that the $i$th sign is positive. Membership in $P_d^+$ therefore guarantees that suitable weights exist. It remains to choose them measurably as $\bm x$ varies.

Since $P_d^+$ is a polytope, it admits a triangulation using only its vertices; see, for example, \cite{barvinok2025course}. Thus, there are finitely many simplices $T_1, \dots, T_N$, each spanned by vertices $v(\bm s^{(j)})$, such that
$$P_d^+=T_1\cup\cdots\cup T_N,$$
and any two simplices intersect, if at all, in a common face.

Fix one simplex and write
$$T_k=\operatorname{conv}\left\{v(\bm s^{(j_0)}), \dots, v(\bm s^{(j_r)})\right\}.$$
Every $\bm x\in T_k$ has unique barycentric coordinates $\lambda_0(\bm x), \dots, \lambda_r(\bm x)$ satisfying
$$\bm x=\sum_{\ell=0}^r\lambda_\ell(\bm x)v(\bm s^{(j_\ell)}), \qquad \lambda_\ell(\bm x)\ge0, \qquad \sum_{\ell=0}^r\lambda_\ell(\bm x)=1.$$
Use these coordinates as the desired probabilities on $T_k$: set $\Psi_{k, j_\ell}(\bm x):=\lambda_\ell(\bm x)$ and set $\Psi_{k, j}(\bm x):=0$ for every other index $j$. Barycentric coordinates are affine functions of $\bm x$, so the map $\Psi_k:T_k\to\Delta_m$ is continuous. Moreover,
$$\sum_{j=1}^m\Psi_{k, j}(\bm x)=1, \qquad \sum_{\substack{1\le j\le m\\s_i^{(j)}=+1}} \Psi_{k, j}(\bm x)=x_i, \quad i\in[d].$$
Consequently, if a random sign vector $\bm S$ takes the value $\bm s^{(j)}$ with probability $\Psi_{k, j}(\bm x)$, then $\p(S_i=+1)=x_i$ for every $i\in[d]$.

We next combine the maps from the different simplices. If $\bm x\in T_k\cap T_\ell$, then $T_k\cap T_\ell$ is a common face. On this face, the barycentric coordinates corresponding to vertices outside the face vanish, while the remaining coordinates are uniquely determined by the vertices of the face. Hence $\Psi_k(\bm x)=\Psi_\ell(\bm x)$. The simplices form a finite closed cover of $P_d^+$, and the continuous maps $\Psi_k$ agree wherever their domains overlap. They therefore paste together to form a continuous map
$$\Psi:P_d^+\longrightarrow\Delta_m.$$
For every $\bm x\in P_d^+$, the vector $\Psi(\bm x)$ gives probabilities on $E_d$ whose positive-sign marginals are exactly $x_1, \dots, x_d$.

It remains to evaluate this continuous choice at $\bm p_u$. Let
$$D:=\{u\in(0, 1):\bm p_u\in P_d^+\}.$$
The set $D$ is Borel because $\bm p_u$ is Borel-measurable and $P_d^+$ is closed. By assumption, $D$ has full Lebesgue measure. Choose $\bm e_1:=(1, 0, \dots, 0)\in\Delta_m$ and define
$$
\bm w(u):=
\begin{cases}
\Psi(\bm p_u), & u\in D, \\
\bm e_1, & u\notin D.
\end{cases}
$$
On $D$, the map $\bm w$ is the composition of the Borel map $u\mapsto\bm p_u$ and the continuous map $\Psi$; outside $D$, it is constant. Since $D$ is Borel, $\bm w$ is Borel-measurable on $(0, 1)$.

Finally, set $w_{\bm s^{(j)}}(u):=w_j(u)$. By construction, the identities in part~(1) hold for every $u\in D$, and hence for a.e.\ $u\in(0, 1)$.

For part~(2), repeat the same construction with $O_d$ in place of $E_d$ and $P_d^-$ in place of $P_d^+$. No other part of the argument changes.
\end{proof}

\begin{proof}[Proof of Theorem \ref{thm:attainability}]
We first prove that \eqref{it:feasible} implies \eqref{it:attain}. Assume that $\bm p_u\in P_d^+$ for a.e.\ $u$. By part~(1) of Lemma~\ref{lem:measurable-selection}, there are Borel-measurable weights $w_{\bm s}^\ast(u)$, $\bm s\in E_d$, satisfying the identities in that statement.

Use these weights in \eqref{eq:maximizer-construction}. For every $u$ at which the identities in Lemma~\ref{lem:measurable-selection} hold, the intervals
$$\bigl(W_{k-1}(u), W_k(u)\bigr], \qquad k=1, \dots, m,$$
form a disjoint partition of $(0, 1]$, and the $k$th interval has length $w_{\bm s^{(k)}}^\ast(u)$. Since $V$ is uniform and independent of $U$,
$$\p(\bm S=\bm s\mid U=u)=w_{\bm s}^\ast(u).$$
Thus $\bm S\in E_d$ a.s., and the marginal identities in part~(1) of Lemma~\ref{lem:measurable-selection} give $\p(S_i=+1\mid U=u)=p_i(u)$ for every $i\in[d]$ and a.e.\ $u$.

Now set $Y_i:=G_i^{-1}(U)$ and $X_i:=S_iY_i$. Since each $G_i^{-1}$ is nondecreasing, the vector
$$(Y_1, \dots, Y_d)=\bigl(G_1^{-1}(U), \dots, G_d^{-1}(U)\bigr)$$
is comonotonic. Also, $|X_i|=Y_i$ for every $i$. Since every sign pattern in $E_d$ has product $+1$,
$$\prod_{i=1}^dX_i =\left(\prod_{i=1}^dS_i\right)\left(\prod_{i=1}^dY_i\right) =\prod_{i=1}^dY_i \quad\text{a.s.}$$

It remains to check that each $X_i$ has the prescribed marginal distribution $F_i$. Fix $i\in[d]$ and a Borel set $B\subseteq[0, \infty)$. Conditioning on $U$ gives
\begin{align*}
\p(X_i\ge0, \ |X_i|\in B) &=\int_{\{u:G_i^{-1}(u)\in B\}}p_i(u)\d u.
\end{align*}
Since $G_i^{-1}(U)$ has density $g_i$, the quantile change-of-variables formula and \eqref{eq:sign-bias} give
\begin{align*}
\p(X_i\ge0, \ |X_i|\in B) &=\int_0^1\id_{\{G_i^{-1}(u)\in B\}}p_i(u)\d u\\
&=\int_B\frac{f_i(y)}{g_i(y)}g_i(y)\d y =\int_B f_i(y)\d y.
\end{align*}
The same calculation for a negative sign uses $1-p_i(u)=f_i(-y)/g_i(y)$ and yields
$$\p(X_i<0, \ |X_i|\in B) =\int_B f_i(-y)\d y.$$
The first identity gives the law of $X_i$ on the nonnegative half-line, and the second gives its law on the negative half-line. Together they show that $X_i\sim F_i$.

We have therefore constructed a random vector with the required marginals. The product identity above proves \eqref{eq:attain-eq}, and hence \eqref{it:attain}. By Lemma~\ref{lem:product-bounds}, the common value in \eqref{eq:attain-eq} is the universal upper bound $M_d^\times$.

For the converse, assume that $(X_1, \dots, X_d)$ satisfies \eqref{eq:attain-eq}. The equality conditions in Lemma~\ref{lem:product-bounds} imply that $|X_1|, \dots, |X_d|$ are comonotonic and that $\prod_{i=1}^dX_i\ge0$ a.s.

Define $U_i:=G_i(|X_i|)$. Since $G_i$ is continuous, each $U_i$ is uniformly distributed. Moreover, $(U_1, \dots, U_d)$ is comonotonic because it is obtained by applying nondecreasing transformations to the comonotonic vector $(|X_1|, \dots, |X_d|)$. Comonotonic random variables with the same uniform distribution are equal a.s.; denote their common value by $U$. The quantile inversion identity then gives
$$|X_i|=G_i^{-1}(U) \quad\text{a.s.\ for every }i\in[d].$$
Define $S_i:=\operatorname{sign}(X_i)$. Under Assumption~\ref{ass:regularity}, $\p(X_i=0)=0$, so the choice of sign at zero is immaterial. Because $\prod_{i=1}^dX_i\ge0$ a.s., the sign vector $\bm S$ has even parity; that is, $\bm S\in E_d$ a.s.

Because $E_d$ is finite, its conditional sign probabilities can be constructed directly from conditional expectations. For each $\bm s\in E_d$, let
$$Z_{\bm s}:=\E\left[\id_{\{\bm S=\bm s\}}\mid\sigma(U)\right].$$
The random variable $Z_{\bm s}$ is $\sigma(U)$-measurable. By the Doob--Dynkin lemma, there is a Borel function $w_{\bm s}^\ast:(0, 1)\to[0, 1]$ such that
$$Z_{\bm s}=w_{\bm s}^\ast(U) \quad\text{a.s.}$$
Conditional expectations preserve nonnegativity, and the indicators $\id_{\{\bm S=\bm s\}}$, $\bm s\in E_d$, sum to one. Consequently, the functions $w_{\bm s}^\ast(u)$ form a probability distribution on $E_d$ for a.e.\ $u$. Here we use that $U$ is uniform: an identity holding a.s.\ after evaluation at $U$ holds for Lebesgue-a.e.\ $u\in(0, 1)$.

We must show that this distribution has the required sign marginals. Fix $i\in[d]$ and define
$$q_i(u):=\sum_{\substack{\bm s\in E_d\\s_i=+1}}w_{\bm s}^\ast(u).$$
Summing the corresponding conditional expectations gives
$$q_i(U)=\E\left[\id_{\{S_i=+1\}}\mid\sigma(U)\right] \quad\text{a.s.}$$
Therefore, for every $t\in[0, 1]$,
\begin{align*}
\int_0^t q_i(u)\d u &=\E\left[\id_{\{U\le t\}}q_i(U)\right] =\p(S_i=+1, \ U\le t)\\
&=\p\bigl(X_i\ge0, \ G_i(|X_i|)\le t\bigr) =H_i(t) =\int_0^t p_i(u)\d u.
\end{align*}
The first equality uses the fact that $U$ is uniform. The second follows by multiplying the conditional-expectation identity for $q_i(U)$ by $\id_{\{U\le t\}}$ and taking expectations. The remaining equalities use $S_i=+1$ iff $X_i\ge0$, the definition of $H_i$, and Lemma~\ref{lem:sign-bias}.

The two indefinite integrals are absolutely continuous and agree at every $t$. Their derivatives therefore agree a.e., so $q_i(u)=p_i(u)$ for a.e.\ $u$. Since there are only finitely many coordinates, these identities hold simultaneously for all $i\in[d]$ outside one null set. On that full-measure set,
$$\bm p_u =\sum_{\bm s\in E_d}w_{\bm s}^\ast(u)v(\bm s) \in P_d^+.$$
Indeed, the $i$th coordinate of the sum is $q_i(u)=p_i(u)$. The right-hand side is a convex combination of vertices of $P_d^+$, and hence belongs to $P_d^+$. This proves \eqref{it:feasible} and completes the proof.
\end{proof}

\begin{proof}[Proof of Corollary \ref{cor:balanced-marginals}]
The density condition and \eqref{eq:sign-bias} imply that $p_i(u)\in[1/d, (d-1)/d]$ for every $i\in[d]$ and a.e.\ $u\in(0, 1)$. For each parity polytope, the nontrivial facet inequalities are given by
$$\sum_{i\in S}x_i-\sum_{i\notin S}x_i\le |S|-1$$
where $S\subseteq[d]$ ranges over subsets of the corresponding parity. Every $\bm x\in[1/d, (d-1)/d]^d$ satisfies these inequalities because the left-hand side is at most
$$|S|\frac{d-1}{d}-(d-|S|)\frac1d=|S|-1.$$
Thus $[1/d, (d-1)/d]^d\subseteq P_d^+\cap P_d^-$, and the result follows from Theorems~\ref{thm:attainability} and~\ref{thm:lower-attainability}.
\end{proof}

\begin{proof}[Proof of Proposition \ref{cor:iid}]
(1) For the upper bound $M_d^{\times}$, write $\mathbf 1_d:=(1, \dots, 1)$ and $\mathbf 0_d:=(0, \dots, 0)$. If $d$ is even, then the all-positive and all-negative sign vectors both belong to $E_d$, so the corresponding parity vertices are $\mathbf 1_d$ and $\mathbf 0_d$. Hence for every $p(u)\in[0, 1]$,
$$p(u)\mathbf 1_d=p(u) \mathbf 1_d+(1-p(u)) \mathbf 0_d\in P_d^+.$$
It follows that $(p(u), \dots, p(u))\in P_d^+$ for every $u$, and $M_d^{\times}$ is sharp by Theorem \ref{thm:attainability}. 

For the lower bound $m_d^{\times}$, We must determine when the diagonal point $p(u) \mathbf{1}_d$ belongs to $P_d^-$. If $d$ is even, every $\bm s\in O_d$ has an odd number of $-1$ entries, hence the number of $+1$ entries is also odd. The vertex coordinate sums $\sum_{i=1}^d v_i(\bm s)$ therefore lie in $\{1, 3, \dots, d-1\}$. For necessity, any point in $P_d^-$ has coordinate sum in $[1, d-1]$. On the diagonal, this sum is $dp(u)$; it follows that $p(u)\in[1/d, (d-1)/d]$. For $d=2$, this gives $p(u)=1/2$, and sufficiency is immediate since $(1/2, 1/2)=\tfrac12(1, 0)+\tfrac12(0, 1)\in P_2^-$. For even $d\ge 4$, the $d$ sign patterns with exactly one $+1$ entry belong to $O_d$ and their indicator-vector average is $(1/d) \mathbf{1}_d\in P_d^-$; the $d$ patterns with exactly one $-1$ entry also belong to $O_d$ and their average is $((d-1)/d) \mathbf{1}_d\in P_d^-$. For $p(u)\in[1/d, (d-1)/d]$,
$$p(u) \mathbf{1}_d=\lambda \frac{d-1}{d} \mathbf{1}_d+(1-\lambda) \frac{1}{d} \mathbf{1}_d, \text{ where } \lambda:=\frac{dp(u)-1}{d-2}\in[0, 1].$$Applying Theorem~\ref{thm:lower-attainability} pointwise in $u$ gives the stated criteria.

(2) Now assume $d$ is odd. In the case of $M_d^{\times}$, for necessity, every $\bm s\in E_d$ has an even number of $-1$ entries, and since $d$ is odd, the number of $+1$ entries is odd. Since $v_i(\bm s)=\id_{\{s_i=+1\}}$, the coordinate sum $\sum_{i=1}^d v_i(\bm s)$ counts the $+1$ entries and is therefore at least $1$. Since $P_d^+$ is the convex hull of these vertices, every $\bm x\in P_d^+$ satisfies $\sum_{i=1}^d x_i\ge 1$. Thus if $p(u)\mathbf 1_d\in P_d^+$, then $p(u)\ge 1/d$. For sufficiency, let $e^{(j)}\in\R^d$ be the $j$-th standard basis vector. Because $d$ is odd, the sign pattern with exactly one positive coordinate and negatives elsewhere has product $+1$, so each $e^{(j)}$ is an admissible parity vertex. Their uniform mixture is
$$\frac1d\sum_{j=1}^d e^{(j)}=\frac1d \mathbf 1_d\in P_d^+.$$
Also $\mathbf 1_d\in P_d^+$ because the all-positive sign pattern is admissible. Hence for every $p(u)\in[1/d, 1]$,
$$p(u)\mathbf 1_d=\lambda \mathbf 1_d+(1-\lambda)\frac1d \mathbf 1_d, \text{ where } \lambda:=\frac{dp(u)-1}{d-1}\in[0, 1].$$
It follows that $p\mathbf 1_d\in P_d^+$ for every $p(u)\in[1/d, 1]$. Applying Theorem \ref{thm:attainability} pointwise in $u$ gives the stated criterion.

In the case of $m_d^{\times}$, every $\bm s\in O_d$ has an odd number of $-1$ entries and $d$ is odd, so the number of $+1$ entries is even. The vertex coordinate sums $\sum_{i=1}^d v_i(\bm s)$ lie in $\{0, 2, \dots, d-1\}$. For necessity, $dp(u)\in[0, d-1]$, so $p(u)\le(d-1)/d$. For sufficiency, the all-negative vector $\bm s=(-1, \dots, -1)$ belongs to $O_d$, giving $\mathbf{0}_d\in P_d^-$. The $d$ patterns with exactly $d-1$ positive entries have coordinate sum $d-1$; their average is $((d-1)/d) \mathbf{1}_d\in P_d^-$. For $p(u)\in[0, (d-1)/d]$,
$$p(u) \mathbf{1}_d= \lambda\frac{d-1}{d} \mathbf{1}_d+\left(1-\lambda\right) \mathbf{0}_d\in P_d^-, \text{ where } \lambda:=\frac{dp(u)}{d-1}.$$
Applying Theorem~\ref{thm:lower-attainability} pointwise in $u$ gives the stated criteria.
\end{proof}

\begin{proof}[Proof of Proposition~\ref{prop:d3-copula}]
Since $|F_1^{-1}(U)| \sim G_1$, the random variable $W$ is uniformly distributed on $[0, 1]$ and is a deterministic function of $U$ alone. Since $I$ is deterministic given $U$, only two even-parity patterns carry positive mass at each level, and the mixing function $s(u)$ selects between them. For $u > u_{0, 1}$ and the maximizer, the active patterns are $(+, +, +)$ and $(+, -, -)$ with conditional weights $s(u)$ and $1-s(u)$; the conditional probability that coordinate~$1$ is positive given $W = w$ is $p_1(w)$, so the unconditional pattern weights at magnitude level $w$ are $p_1(w)s(u) = w_{+++}(w)$ and $p_1(w)(1-s(u)) = w_{+--}(w)$. Solving for $s$ and matching with~\eqref{eq:d3-even-weights} yields the upper branch of~\eqref{eq:d3-maximizer-mixing}; the same argument on $(0, u_{0, 1}]$ gives the lower branch. The even-parity feasibility $\left(p_1(w), p_2(w), p_3(w)\right)\in P_3^+$ for a.e.\ $w$ is equivalent to $s(u) \in [0, 1]$ on both halves, and the marginal property $U_i \sim \U[0, 1]$ follows from Theorem~\ref{thm:attainability}. The minimizing copula is obtained from~\eqref{eq:d3-maximizer-copula} by swapping the role of $I$ in the third coordinate, which sends every sign pattern from $E_3$ to $O_3$ and ensures $X_1 X_2 X_3 \le 0$ a.s.; the marginal property follows from Theorem~\ref{thm:lower-attainability}. The derivation of~\eqref{eq:d3-minimizer-mixing} follows the same pattern as for the maximizer. The odd-parity feasibility $\left(p_1(w), p_2(w), p_3(w)\right)\in P_3^-$ for a.e.\ $w$ is equivalent to $s(u) \in [0, 1]$ on both halves.
\end{proof}

\begin{proof}[Proof of Proposition \ref{prop:recursive-parity}]
Suppose $\bm p\in P_d^+$ and $p_1\in(0, 1)$. Then $\bm p=\sum_{\bm s\in E_d}w_{\bm s}v(\bm s)$ with $w_{\bm s}\ge 0$ and $\sum w_{\bm s}=1$. Group the sum by $s_1$: the patterns with $s_1=+1$ contribute total weight $p_1$ and, after conditioning, the residual vector $(s_2, \dots, s_d)$ has product $+1$ (even parity in $d-1$ dimensions); similarly, $s_1=-1$ patterns contribute weight $1-p_1$ and the residual has odd parity. Where $q_j^+:=\sum_{\bm s:s_1=+1, \,s_j=+1}w_{\bm s}/p_1$ and $q_j^-:=\sum_{\bm s:s_1=-1, \,s_j=+1}w_{\bm s}/(1-p_1)$. Then $\bm q^+\in P_{d-1}^+$, $\bm q^-\in P_{d-1}^-$, and the law of total probability gives~\eqref{eq:recursive-decomposition}.

Conversely, given $\bm q^+\in P_{d-1}^+$ and $\bm q^-\in P_{d-1}^-$ satisfying~\eqref{eq:recursive-decomposition}, write $\bm q^+=\sum_{\bm r\in E_{d-1}}\alpha_{\bm r}v(\bm r)$ and $\bm q^-=\sum_{\bm r\in O_{d-1}}\beta_{\bm r}v(\bm r)$. Setting $w_{(+1, \bm r)}:=p_1\alpha_{\bm r}$ for $\bm r\in E_{d-1}$ and $w_{(-1, \bm r)}:=(1-p_1)\beta_{\bm r}$ for $\bm r\in O_{d-1}$ gives a valid even-parity distribution on $E_d$ with marginal biases $\bm p$, so $\bm p\in P_d^+$.

The boundary cases $p_1\in\{0, 1\}$ follow by taking $w_{\bm s}=0$ for all patterns with the excluded value of $s_1$. The odd-parity analogue is obtained by interchanging the residual parity requirement.
\end{proof}
\end{document}